\documentclass{article}
\usepackage{graphicx} 

\usepackage{amsmath}
\usepackage{amsthm}
\usepackage{appendix}
\usepackage{hyperref}

\usepackage[disable]{todonotes}
\usepackage{virginialake}
\vllineartrue

\usepackage{enumerate}

\usepackage{proof}

\usepackage{tikz}
\usepackage{tikz-cd}
\usetikzlibrary{patterns,arrows, topaths, calc, positioning}
\usepackage{algorithm}%
\usepackage{algorithmicx}%
\usepackage{etoolbox}

\usepackage[capitalise]{cleveref}
\usepackage{thmtools} 
\usepackage{authblk}

\newcommand{\anupam}[1]{\todo{AD: #1}}
\newcommand{\tikhon}[1]{\todo{TP: #1}}

\renewcommand{\neg}[1]{#1^{\perp}}
\newcommand{\negneg}[1]{#1^{\perp\perp}}

\newcommand{\df}{:=}
\newcommand{\bnf}{::=}

\newcommand{\IH}{\mathit{IH}}

\newcommand{\parto}{\rightharpoonup}

\renewcommand{\epsilon}{\varepsilon}
\renewcommand{\emptyset}{\varnothing}

\newcommand{\UniNumA}{u}
\newcommand{\UniNumB}{v}
\newcommand{\UniNumC}{w}

\newcommand{\UniSetA}{U}

\newtheorem{theorem}{Theorem}[section]

\newtheorem{proposition}[theorem]{Proposition}
\newtheorem{lemma}[theorem]{Lemma}
\newtheorem{corollary}[theorem]{Corollary}
\newtheorem{fact}[theorem]{Fact}

\theoremstyle{definition}
\newtheorem{definition}[theorem]{Definition}
\newtheorem{convention}[theorem]{Convention}

\theoremstyle{remark}
\newtheorem{remark}[theorem]{Remark}
\newtheorem{example}[theorem]{Example}

\newcommand{\Var}{\mathsf{Var}}
\newcommand{\FVar}{\mathsf{FV}}
\newcommand{\Prop}{\mathsf{Prop}}
\newcommand{\Ord}{\mathsf{Ord}}
\newcommand{\Fm}{\mathsf{Fm}}

\newcommand{\muord}[1]{\mu^{#1}}
\newcommand{\nuord}[1]{\nu^{#1}}

\newcommand{\focusedNeg}{\Uparrow}
\newcommand{\focusedPos}{\Downarrow}
\newcommand{\emptyZone}{\mathord{\cdot}}

\newcommand{\lr}[1]{#1_l}
\newcommand{\rr}[1]{#1_r}

\newcommand{\MALL}{\mathsf{MALL}}
\newcommand{\muMALL}{\mu\MALL}
\newcommand{\muMALLord}[1]{\muMALL_{#1,#1}}

\newcommand{\ind}{\mathsf{ind}}
\newcommand{\coind}{\mathsf{coind}}

\newcommand{\cf}{\mathit{cf}}
\newcommand{\muMALLordcf}[1]{\muMALL_{#1,#1}^{\cf}}
\newcommand{\muMALLordcut}[2]{\muMALLord #1^{#2}}

\newcommand{\muMALLFord}[1]{\muMALL\mathsf{F}_{#1,#1}^\cf}

\newcommand{\identity}{\mathsf {id}}
\newcommand{\cut}{\mathsf{cut}}
\newcommand{\dcut}[1]{#1\text{-}\cut}
\newcommand{\store}{\mathsf{s}}
\newcommand{\release}{\mathsf{r}}
\newcommand{\decide}{\mathsf{d}}

\newcommand{\seqar}{\Rightarrow}

\newcommand{\proves}{\vdash}

\newcommand{\inv}[1]{#1^{-1}}

\newcommand{\dom}{\mathop{\mathrm{dom}}}
\newcommand{\pair}[2]{\langle #1, #2 \rangle}
\newcommand{\UCF}{\mathtt{U}} 

\newcommand{\MinMach}{\mathtt{M}}
\newcommand{\INC}{\mathtt{INC}}
\newcommand{\JZDEC}{\mathtt{JZDEC}}
\newcommand{\reach}[1]{\to_{#1}}

\newcommand{\mred}{\mathrm{m}}
\newcommand{\Tred}{\mathrm{T}}

\newcommand{\KleeneO}{\mathcal{O}}
\newcommand{\ordO}[1]{\vert #1 \vert_{\KleeneO}}
\newcommand{\lessO}{<_{\KleeneO}}

\newcommand{\ips}[1]{\mathrm{ips}{( #1 )}} 
\newcommand{\TrueIPS}[1]{\mathrm{Tr}{( #1 )}}
\newcommand{\Witn}{\mathrm{Wit}}

\newcommand{\ld}{\mathop{\mathrm{ld}}} 
\newcommand{\natframe}[1]{\mathbin{\vcenter{\hbox{\resizebox{!}{0.075in}{#1}}}}}
\newcommand{\natsum}{\natframe{$\boxplus$}}
\newcommand{\natprod}{\natframe{$\boxtimes$}}
\newcommand{\bignatsum}{\vcenter{\hbox{\resizebox{\widthof{$\bigoplus$}}{!}{$\boxplus$}}}}

\newcommand{\valuation}[2]{\vert #1 \vert^{#2}}
\newcommand{\rank}{\mathrm{rk}}
\newcommand{\rk}[1]{\rank(#1)}

\newcommand{\mult}{\mathop{\mathrm{mult}}}

\newcommand{\Resource}{R}
\newcommand{\aug}{\mathrm{aug}}
\newcommand{\lock}{\mathrm{lock}}
\newcommand{\tup}{\mathrm{tup}}
\newcommand{\ins}{\mathit{ins}} 
\newcommand{\encIns}[1]{[ #1 ]} 

\newcommand{\VARacc}{\mathit{acc}}
\newcommand{\VARkey}{\mathit{key}}
\newcommand{\VARsym}{c}

\newcommand{\FORAcc}{\mathit{Acc}}
\newcommand{\FORBase}{\mathit{Base}}
\newcommand{\FORComp}{\mathit{Comp}}
\newcommand{\FORInd}{\mathit{Ind}}

\newcommand{\FORNext}{\mathit{Next}}
\newcommand{\FORPass}{\mathit{Pass}}
\newcommand{\FORStep}{\mathit{Step}}

\title{Wider systems for linear logic with fixed points: proof theory and complexity}
\author[1]{Anupam Das}
\author[2]{Tikhon Pshenitsyn}
\affil[1]{University of Birmingham, UK}
\affil[2]{Steklov Mathematical Institute of RAS, Russia}

\date{}

\begin{document}

\maketitle

\begin{abstract}
    We investigate infinitary wellfounded systems for linear logic with fixed points, with transfinite branching rules indexed by some closure ordinal $\alpha$ for fixed points. 
    Our main result is that provability in the system for some computable ordinal $\alpha$ is complete for the $\omega^{\alpha^\omega}$ level of the hyperarithmetical hierarchy. 

To this end we first develop proof theoretic foundations, namely cut elimination and focussing results, to control both the upper and lower bound analysis. Our arguments employ a carefully calibrated notion of formula rank, calculating a tight bound on the height of the (cut-free) proof search space.
\end{abstract}

\section{Introduction}

\emph{Fixed points} abound in logic and computation, providing a robust way to formalise induction and recursion.
The last 15-20 years has seen renewed interest in fixed points from the point of view of \emph{structural proof theory}, building a finer picture of their logical behaviour, e.g.\ see \cite{Baelde-Miller,baelde2012least,BaeldeDoumaneSaurin2016CSL,BaeldeDoumaneKuperberg2022LICS,DeSaurin2019TABLEAUX,DePellissierSaurin2021PPDP,Jafarrahmani2021CSL,DeJafarrahmaniSaurin2022FSTTCS} in just the setting of linear logic. 
These investigations cover several proposals for reasoning about fixed points, including systems with explicit (co)induction rules, cyclic or non-wellfounded proofs, and systems with infinitely branching rules. 
The setting of {linear logic} allows us to compare these methodologies in a robust way, largely free from artefacts of, say, classical or intuitionistic logic, or first-order theories, such as arithmetic.
In particular recent works, such as \cite{das_et_al:LIPIcs.FSCD.2022.20,DeJafarrahmaniSaurin2022FSTTCS,DasDS23}, have improved our understanding of how various methodologies compare in terms of the theorems they admit, and the complexity of reasoning therein.

In this paper we extend this line of work to systems whose branching is indexed by an arbirtary (computable) ordinal. 
Previous work such as \cite{DeJafarrahmaniSaurin2022FSTTCS,DasDS23} only considered $\omega$-branching which are suitable for only continuous models, where definable operators reach their fixed points in $\leq \omega$ many steps.
General models, e.g.\ of the arithmetical/first-order or modal $\mu $-calculus, may admit arbitrary closure ordinals (see, e.g., \cite{GurevichShelah1986,Moellerfeld2002Thesis,AfshariLeigh2013CSL}).
For instance $\mu x\,  (1 + \nu y\,  (x \times y))$ represents the class of all (countable) well-founded trees, when interpreted as a subset of Cantor space.
Apart from the expanded scope of our investigations, this work serves as a further testbed for the tools developed in structural proof theory for handling fixed points in the last couple decades.

\subsection{Contribution}
We investigate the proof theory and complexity of the systems $\muMALLord \alpha$ for (multiplicative additive) linear logic with fixed points, first proposed by De in \cite{De2022Thesis}, parametrised by an arbitrary closure ordinal $\alpha$ for fixed points.
We prove \emph{cut-elimination} and \emph{focussing} results, by exploiting an ordinal measure of formulas, \emph{rank}, and adapting existing methods.
Our definition of rank is carefully calibrated in order to tightly bound the height of the cut-free proof search space, unlike previous approaches, e.g.\ \cite{AlberucciKS14,De2022Thesis}. 
The main application of these foundational developments is a classification of the complexity of system $\muMALLord\alpha$ for computable $\alpha$: it is complete for the $\omega^{\alpha^\omega}$ level of the hyperarithmetical hierarchy (under Turing reductions).
The upper bound follows cleanly from the aforementioned property of ranks, under a natural encoding of cut-free proof search. For the lower bound we encode the problems of \emph{true computable infinitary sentences} complete for each level of the hyperarithmetical hierarchy, exploiting our focussing result to show adequacy.



\subsection{Structure of the paper}
The remainder of this paper is structured as follows.
In \cref{sec:prelims} we define the systems $\muMALLord\alpha$ and state our main result, in particular recapping representation issues around computable ordinals.
In \cref{sec:pt} we present our foundational proof theoretic results for $\muMALLord\alpha$, cut-elimination and focussing, for all $\alpha$, assuming termination of cut-free proof search.
The latter is addressed in \cref{sec:ranks}, where give our notion of formula rank and prove its tightness.
In \cref{sec:upper-bd,sec:lower-bd} we prove the upper and lower bounds, respectively, for our logics.
Finally in 
\cref{sec:concs} we present some further pespectives and concluding remarks.

\section{Preliminaries}
\label{sec:prelims}

In this section we shall recap systems for linear logic (with fixed points), and present the ordinal branching systems that comprise the main focus of this work.

Throughout let us fix a set $\Var$ of \textbf{variables}, written $x,y$ etc., and a disjoint set $\Prop$ of \textbf{propositional symbols}, written $p,q$ etc.

\subsection{Multiplicative additive linear logic with fixed points}
We assume basic familiarity with multiplicative additive linear logic $\MALL$, for which \cite{Girard1987TCS,sep-logic-linear} are useful introductions.

\textbf{$\muMALL$ formulas}, written $A,B,C $ etc., are generated as follows,
\begin{equation}
    \label{eq:muMALL-fmlas}
    \begin{array}{rcc@{\ \mid \ }c@{\ \mid \ }c@{\ \mid \ }c@{\ \mid \ }c@{\ \mid \ }c@{\ \mid \ }c}
A,B, \dots  & \bnf &  
x & p &  1 & 0 &  A \vlte B & A \vlor B & \mu x  A \\
& & & \neg p &  \bot & \top &  A \vlpa B & A \vlan B &  \nu x A
\end{array}
\end{equation}
where $x \in \Var$ and $p \in \Prop$.
The set of \textbf{free variables} of a formula $A$, written $\FVar(A)$, is defined as expected, construing $\mu$ and $\nu$ as variable binders.

The $\MALL$ connectives (without $\mu,\nu$) decompose classical ones each into two variants, \textbf{additive} and \textbf{multiplicative} according to the following classification:\footnote{The fixed points are not usually assigned such a designation, but both $\mu$ and $\nu$ could be considered additive.}
\begin{center}
    \begin{tabular}{c|c|c}
         & additive & multiplicative \\
         \hline
    false & $0$ & $\bot$ \\
    true & $\top$ & $1$ \\
    disjunction & $\vlor$ & $\vlpa$ \\
    conjunction & $\vlan$ & $\vlte$
    \end{tabular}
\end{center}
A \textbf{literal} is a formula of the form $p$ or $\neg p$ where $p \in \Prop$.
We extend the notation $\neg A $ to all formulas $A$, not just propositional symbols, by construing connectives in the same column in \cref{eq:muMALL-fmlas} as dual:\todo{say more about variables being self-dual}
\[
\begin{array}{r@{\ := \ }l}
     \neg x & x \\
     \negneg p & p \\
     \neg 1 & \bot \\
     \neg \bot & 1 \\
     \neg 0 & \top \\
     \neg \top & 0
\end{array}
\qquad
\begin{array}{r@{\ := \ }l}
     \neg {(A \vlte B)} & \neg A \vlpa \neg B \\
     \neg {(A \vlpa B)} & \neg A \vlte \neg B \\
     \neg{(A\vlor B)} & \neg A \vlan \neg B\\
     \neg {(A\vlan B)} & \neg A \vlor \neg B \\
     \neg{(\mu x A)} & \nu x \neg A \\
     \neg {(\nu x A)} & \mu x \neg A
\end{array}
\]
Note that negation does not change the additive/multiplicative status of a formula's main connective.
The clause setting $\neg x = x$ is justified when $x$ is bound by a $\mu$ or a $\nu$ in a formula.
That is, whenever $A$ is closed, $\neg A$ will represent its negation in any systems and semantics we consider.
\anupam{Is this okay now? Previous commentary commented above.}

\begin{figure}[t]
    \[
    \begin{array}{ccccc}
         \vlinf{\identity}{}{A, \neg A}{}
         & 
         \vlinf{1}{}{1}{}
         & 
         {\color{gray}
         \left(
         \begin{array}{c}
              \text{no rule}  \\
              \text{for $0$}
         \end{array}
         \right)}
         &
         \vliinf{\vlte}{}{\Gamma, \Delta, A \vlte B}{\Gamma, A}{\Delta, B}
         & \vlinf{\vlor}{}{\Gamma, A_0 \vlor A_1}{\Gamma, A_i}
         \\
         \noalign{\smallskip}
         \vliinf{\cut}{}{\Gamma, \Delta}{\Gamma, A}{\Delta, \neg A}
         & \vlinf{\bot}{}{\Gamma, \bot}{\Gamma}
         & \vlinf{\top}{}{\Gamma, \top}{}
         & \vlinf{\vlpa}{}{\Gamma, A \vlpa B}{\Gamma, A, B}
         & \vliinf{\vlan}{}{\Gamma, A\vlan B}{\Gamma, A}{\Gamma, B}
    \end{array}
    \]
    \caption{The system $\MALL$ (without fixed points), where $i \in \{0,1\}$.}
    \label{fig:MALL}
\end{figure}

We shall work with extensions of a usual one-sided sequent system for $\MALL$, with sequents as multisets, in \cref{fig:MALL}. From here the basic system for $\muMALL$ formulas is obtained by appropriate extremal fixed point principles:

\begin{definition}
    [System $\muMALL$]
The system $\muMALL$ is the extension of $\MALL$, cf.~\cref{fig:MALL}, by the rules:
\[
\vlinf{\mu}{}{\Gamma , \mu x A(x)}{\Gamma, A(\mu x A(x))}
\qquad
\vliinf{\nu}{}{\Gamma, \nu x A(x)}{\Gamma , B}{\neg B, A(B)}
\]
\end{definition}

\begin{convention}
    [Two-sided sequents]
    We shall systematically use \emph{two-sided} notation, writing $\Gamma \seqar \Delta$ formally for the sequent $\neg \Gamma, \Delta$, where $\neg \Gamma := \{\neg A : A \in \Gamma\}$.
    This is purely so that sequents may be read suggestively, making explicit the stated implication that is formally coded by De Morgan duality.
    In the same vein we may write, say $\lr \mu$ or $\rr \nu$ for the same rule $\nu$, depending on `which side' of the sequent the principal formula is on.
\end{convention}

\begin{remark}
    [(Co)induction principles]
    \label{co-induction-principles}
    The rules above, under duality, indicate that $\mu $ and $\nu $ define least and greatest fixed points, respectively.



In particular, the following \emph{induction} and \emph{coinduction} rules, respectively, are just special cases of the $\nu$-rule, by De Morgan duality:
\[
\vlinf{\ind}{}{\mu x A(x) \seqar B}{A(B)\seqar B}
\qquad
\vlinf{\coind}{}{B \seqar \nu x A(x)}{B \seqar A(B)}
\]
\end{remark}

The basic theory behind $\muMALL$ is well-established; the reader may consult \cite{Baelde-Miller,baelde2012least} for further background.
Several other systems over the language of $\muMALL$ have been proposed in the literature, in particular \emph{non-wellfounded} and \emph{circular} ones \cite{BaeldeDoumaneSaurin2016CSL,BaeldeDoumaneKuperberg2022LICS}. See \cite{das_et_al:LIPIcs.FSCD.2022.20,DasDS23} for a comparison of earlier results comparing these systems at the level of provability and complexity.

\subsection{Wellfounded systems with ordinal branching}
Among the systems considered in \cite{De2022Thesis,DeJafarrahmaniSaurin2022FSTTCS,DasDS23} were wellfounded systems with (some) ordinal branching.
The general investigation of such systems, at large, is the main focus of this work.
We shall follow the naming conventions of \cite{DasDS23}.

\begin{definition}
    [System $\muMALLord \alpha$]
    Fix an ordinal $\alpha$.
    The language of the system $\muMALLord \alpha$
     expands the syntax of $\muMALL$ formulas as follows,
\[
A,B, \dots \quad \bnf \quad
\cdots \ \mid \ \muord \beta x A \ \mid \ \nuord \beta x A
\]
for $\beta \leq \alpha$.
We extend the notation $\neg A$ for formulas $A$ in this extended syntax by setting $\neg {(\muord \beta x A)} := \nuord\beta x \neg A$ and $\neg{(\nuord \beta x A)} := \muord \beta x \neg A$.

We identify $\mu x A$ and $\nu x A$ with $\muord\alpha x A$ and $\nuord \alpha x A $ respectively.
The system $\muMALLord \alpha$ is the extension of $\MALL$ by the rules,
\begin{equation}
    \label{eq:muord-nuord-rules}
    \vlinf{\muord \beta }{\gamma <\beta}{\Gamma, \muord \beta x A(x)}{\Gamma, A(\muord \gamma x A(x))}
\qquad
\vlinf{\nuord \beta}{}{\Gamma, \nuord \beta X A(X)}{ \{ \Gamma, A(\nuord \gamma X A(X)) \}_{\gamma < \beta}  }
\end{equation}
for all $\beta \leq \alpha$.
\end{definition}
The system $\muMALLord\alpha$, in this generality, was first proposed by De in \cite[Section~5.2]{De2022Thesis}, building on earlier work with only $\omega$-branching \cite{Jafarrahmani2021CSL}. 
In particular he gives a soundness and completeness result for certain phase models.
While consideration of these semantics are beyond the scope of this work, we shall make due comparisons to \cite{De2022Thesis} inline.

We assume that $\alpha$ is computable because otherwise it is not clear how to assign G\"odel numbers to formulae of the form $\muord \beta x\,  A(x)$ when considering the provability problem for a logic. 
Once a computable wellorder of the order type $\alpha$ is fixed, one can denote $\beta$ by its notation according to that wellorder.

\begin{remark}
    [Approximant reading]
    Note that the $\nuord \beta$ rules have (possibly) infinitely many premisses, indexed by all $\gamma<\beta$.
    The rules above are consistent with the (informal) interpretation of $\muord \beta x A(x) $ as $ 0 \vlor A(0) \vlor A(A(0)) \cdots A^\beta (0)$ and $\nuord \beta x A(x)$ as $\top \vlan A(\top) \vlan A(A(\top)) \vlan \cdots \vlan A^\beta(\top)$ induced by the \emph{approximant} characterisation of least and greatest fixed points of monotone operators.
    Here the additives, $(0,\top, \vlor, \vlan)$, model the lattice structure.
\end{remark} 
Let us expand on this intuition via a couple examples:

\begin{example}
    [Additive units]
    $\mu x \, x $ and $\nu x\, x $ are actually equivalent to the additive units, $0$ and $\top$ respectively, in any system $\muMALLord\alpha$. 
    By duality, it suffices to show only $\nu x\,  x $ equivalent to $\top$.
    $\nu x\, x \seqar \top$ is already a theorem of $\MALL$ (over the expanded language), by the $\top$ rule.
    To prove $\top \seqar  \nu x\, x $, it suffices to prove $\Gamma, \nuord\beta x \, x $ for all $\beta \leq \alpha$, for which we proceed by induction on $\beta$:
    \[
    \vlinf{\nuord \beta}{}{\Gamma, \nuord \beta x\, x }{
    \left\{
    \vltreeder{\IH}{\Gamma, \nuord \gamma x\, x }{\quad }{}{\quad }
    \right\}_{\gamma<\beta}
    }
    \]
\end{example}

This equivalence holds in other systems over the syntax of $\muMALL $ too (see, e.g., \cite{DasDS23}).
    Note that, here, we may also readily see that $\muord 0 x A$ and $\nuord 0 x A$ are equivalent to $0$ and $\top$ respectively.
However note that the systems $\muMALLord\alpha$ do not, in general, derive the (co)induction principles, cf.~\cref{co-induction-principles}, see \cite[Section~4.2]{Jafarrahmani2021CSL}.

\todo[inline]{Maybe some further examples and comparison with other systems? Can we say more about induction?}

\begin{example}
    [Monotonicity]
    We can show that $\muord \beta x A(x)$ is increasing and $\nuord  \beta x A(x)$ is decreasing in $\beta$.
    Namely for $\gamma \leq \beta$ we have:
    \[
    \vlinf{\lr \mu}{}{\muord \gamma x A(x) \seqar \muord \beta x A(x) }{
    \left\{
    \vlderivation{
    \vlin{\rr \mu}{}{A(\muord \delta x A(x)) \seqar \muord \beta x A(x) }{
    \vlin{\identity}{}{A(\muord \delta x A(x)) \seqar A(\muord \delta x A(x))}{\vlhy{}}
    }
    }
    \right\}_{\delta<\gamma}
    }
    \]
    The case for $\nu$ is identical, by De Morgan duality.
\end{example}



\subsection{Statement of main result}

The main question we ask in this work is about the algorithmic complexity of theoremhood in $\muMALLord \alpha$: is a $\muMALL$ sequent $\Gamma$ provable in $\muMALLord \alpha$?
We restrict this problem to only $\muMALL$ input sequents for two reasons: (i) so that the theorems of all logics are over the same language; and, (ii) to avoid representation issues of the input sequent when ordinal annotations appear (recall that $\mu$ is identified with $\muord \alpha$ in $\muMALLord \alpha$).
Our main result comprises a tight answer to this question for appropriate finitely representable ordinals $\alpha$:



\begin{theorem}\label{theorem:complexity-main}
    $\muMALLord \alpha$ is $\Sigma^0_{\omega^{\alpha^\omega}}$-complete, for computable ordinals $\alpha \ge \omega$.
\end{theorem}

The complexity class $\Sigma^0_{\omega^{\alpha^\omega}}$ is a level of the \emph{hyperarithmetical hierarchy}---a classification of $\Delta^1_1$-sets extending the arithmetical hierarchy. 
One can obtain a complete set at the $n$\textsuperscript{th} level of the arithmetical hierarchy by $n$-fold application of the Turing jump to $\emptyset$. 
Likewise, a complete set at the $\beta$\textsuperscript{th} level of the hyperarithmetical hierarchy is obtained by iterating the Turing jump $\beta$ times, taking \emph{effective unions} at limit stages. 
The hyperarithmetical hierarchy is construed under \emph{Turing} reductions, not many-one reductions. This is explained further in the sequel.

The remainder of this paper is devoted to proving \cref{theorem:complexity-main}, along with some necessary proof theoretic foundations.
Before that, let us recall sufficient effective descriptive set theory for the reader to make sense of the theorem statement above, not least to make precise our notion of \emph{computable reduction} therein.

\subsection{Representing recursive ordinals}


The \emph{hyperarithmetical hierarchy} is an effective version of the Borel hierarchy of $\Delta^1_1$-sets. 
It is usually defined using Kleene's $\KleeneO$, a system of notations for all recursive ordinals \cite{Kleene38}. In this section we recall these definitions, but point the reader to textbooks such as \cite{AshK2000,Rogers67} for a more detailed account. 
The reader familiar with the hyperarithmetical hierarchy may safely skip this subsection, only taking note of \cref{dfn:hyperarith-hier,dfn:hyperarith-completeness-hardness}.

\subsubsection{Computable ordinals and Kleene's $\KleeneO$}
Fix some enumeration of partial computable one-variable functions that use a set $X \subseteq \omega$ as an oracle; let $\varphi_e^X$ denote the $e$-th such function. We denote $\varphi_e^\emptyset$ by $\varphi_e$. A set $A \subseteq \omega$ is \textbf{computable} in $B \subseteq \omega$ if there is $e \in \omega$ such that
\[
\varphi_e^B(n) = 
\begin{cases}
    1, & n \in A \\
    0, & n \notin A
\end{cases}
\]
If $A$ is computable in $B$, we also say that $A$ is \textbf{Turing-reducible to $B$} and denote this by $A \le_\Tred B$. If $A \le_\Tred B$ and $B \le_\Tred A$, we say that $A$ and $B$ are Turing-equivalent, $A \equiv_\Tred B$.

$A$ is \textbf{computably enumerable} in $B$ if there is $e \in \omega$ such that
\[
\varphi_e^B(n) = 
\begin{cases}
    1, & n \in A \\
    \text{undefined}, & n \notin A
\end{cases}
\]
A folklore fact is that $A$ is computable in $B$ iff $A$ and $\omega \setminus A$ are computably enumerable in $B$ \cite{Rogers67}.

\begin{definition}
[Ordinal notations]
    Define a partial function $\ordO{\cdot}$ mapping certain natural numbers (called \emph{notations}) to ordinals and a strict well partial order $\lessO$ on notations by transfinite recursion as follows:
    \begin{enumerate}
        \item The only notation for $0$ is $1$, i.e.~$\ordO{1} = 0$. 
        \item The notations for $\alpha+1$ are those of the form $2^a$ where $a$ is a notation for $\alpha$, i.e.~$\ordO{2^a} = \ordO{a}+1$. For any $b$, $b \lessO 2^a$ iff $b \lessO a$ or $b=a$.
        \item Let $\alpha$ be limit and $e$ be the index of a total computable function such that $\varphi_e(0) \lessO \varphi_e(1) \lessO \ldots$ and $\sup\limits_{n \in \omega} \ordO{\varphi_e(n)} = \alpha$. Then, $3 \cdot 5^e$ is a notation for $\alpha$, i.e.~$\ordO{3\cdot 5^e} = \alpha$. For any $b$, $b \lessO 3 \cdot 5^e$ iff $b \lessO \varphi_e(n)$ for some $n \in \omega$.
    \end{enumerate}
    Finally, \textbf{Kleene's $\KleeneO$} is the domain of $\ordO{\cdot}$ as defined above.
\end{definition}

\begin{example}
    A finite ordinal $n \in \omega$ receives the unique notation $2 \uparrow\uparrow n = \underbrace{2^{2^{\ldots^2}}}_{n~\text{times}}$. The ordinal $\omega$ is the least ordinal with infinitely many notations. For example, if $e_1$ is an index of a function $\varphi_{e_1}(n) = 2 \uparrow\uparrow n$, then $3 \cdot 5^{e_1}$ is a notation for $\omega$; but also if $\varphi_{e_2}(n) = 2 \uparrow\uparrow (2n)$, then $3 \cdot 5^{e_2}$ is a notation for $\omega$ as well. Then, both $2^{3 \cdot 5^{e_1}}$ and $2^{3 \cdot 5^{e_2}}$ are notations for $\omega+1$. Thus, the tree of notations looks as follows, branching at each limit step:
    \begin{center}
        \tikz[baseline=.1ex]{
            \def\VLEN{1mm}
            \def\HLEN{5mm}
            \def\superscripts{-1.1cm}
            \node[rectangle] (Not0) at (0,0) {1};
            \node[rectangle,right=\HLEN of Not0] (Not1) {$2$};
            \node[rectangle,right=\HLEN of Not1] (Not2) {$4$};
            \node[rectangle,right=\HLEN of Not2] (Not3) {$16$};
            \node[rectangle,right=\HLEN of Not3] (Not4) {$65536$};
            \node[rectangle,right=\HLEN of Not4] (Not5) {$\cdots$};
            \foreach \i in {0,...,4}
            {
                \pgfmathtruncatemacro\Si{\i+1}
                \draw[->] (Not\i.east) -- (Not\Si.west) ;
            }
            \node[rectangle,above right=\VLEN and \HLEN of Not5] (OmegaA) {$3\cdot 5^{e_1}$};
            \node[rectangle,right=\HLEN of Not5] (OmegaB) {$3\cdot 5^{e_2}$};
            \node[rectangle,below right=0mm and 0mm of Not5] (OmegaC) {$\vdots$};
            \node[rectangle,right=\HLEN of OmegaA] (OmegaA1) {$2^{3\cdot 5^{e_1}}$};
            \node[rectangle,right=\HLEN of OmegaB] (OmegaB1) {$2^{3\cdot 5^{e_2}}$};
            \node[rectangle,right=\HLEN of OmegaA1] (OmegaA2) {$\cdots$};
            \node[rectangle,right=\HLEN of OmegaB1] (OmegaB2) {$\cdots$};
            \draw[->] (Not5.east) -- (OmegaA.west) ;
            \draw[->] (Not5.east) -- (OmegaB.west) ;
            \draw[->] (OmegaA.east) -- (OmegaA1.west) ;
            \draw[->] (OmegaB.east) -- (OmegaB1.west) ;
            \draw[->] (OmegaA1.east) -- (OmegaA2.west) ;
            \draw[->] (OmegaB1.east) -- (OmegaB2.west) ;
            \foreach \i in {0,...,4}
            {
            \path let \p1 = (Not\i) in node  at (\x1,\superscripts) {\color{gray} $\i$};
            }
            \path let \p1 = (OmegaA) in node  at (\x1,\superscripts) {\color{gray} $\omega$};
            \path let \p1 = (OmegaA1) in node  at (\x1,\superscripts) {\color{gray} $\omega+1$};
        }
    \end{center}
    The particular choice of the numbers $2,3,5$ in the definition of Kleene's $\KleeneO $ is not important, they could be replaced by any other distinct prime numbers or other appropriate coding mechanism.
\end{example}

\subsubsection{Hyperarithmetical hierarchy}

The complexity result of \cref{theorem:complexity-main} places us in the realm of the \emph{hyperarithmetical hierarchy}, the transfinite extension of arithmetical hierarchy, which provides a fine-grained classification of $\Delta^1_1$-sets. 

\begin{definition}
[Turing jumps]
    For $A \subseteq \omega$, the \textbf{Turing jump} of $A$ is the set:
    \[
        A^\prime = \{n \in \omega \mid \varphi_n^A(n)~\text{is defined}\}.
    \]
\end{definition}

\begin{definition}
[Hyperarithmetical hierarchy]
\label{dfn:hyperarith-hier}
For each $a \in \KleeneO$, we define the set $H(a) \subseteq \omega$ by induction on $\lessO$ as follows.
\begin{enumerate}
    \item $H(1) \df \emptyset$;
    \item $H(2^a) \df (H(a))^\prime$;
    \item $H(3\cdot5^e) = \{\pair{b}{n} \mid b \lessO 3\cdot 5^e,\ n \in H(b)\}$.
\end{enumerate}
Let $\Sigma^0_0 = \Pi^0_0$ be the collection of all computable subsets of $\omega$. For each non-zero ordinal $\alpha$, we set:
\[
\begin{array}{r@{\ \df \ }l}
     \Sigma^0_{\alpha} & {\left\{
S \subseteq \omega \mid \exists a\in \KleeneO\ \ordO{a} = \alpha,\
S \le_{\Tred} H \left( a \right)
\right\}} \\
  \Pi^0_{\alpha}   & {\left\{
\omega \setminus S \mid
S \in \Sigma^0_\alpha
\right\}} 
\end{array}
\]
\end{definition}
An important caveat is that $\Sigma^0_\alpha$ is defined via Turing reducibility $\le_{\Tred}$. If it was defined via many-one reducibility $\le_{\mred}$, then, in general, there would be no complete set in $\Sigma^0_\alpha$: as proved in \cite{Moschovakis66}, if $\alpha$ is limit and is not of the form $\beta+\omega$, then, for each $a$ such that $\ordO{a}=\alpha$, there is $b$ such that $\ordO{b}=\alpha$ and $H(a) <_{\mred} H(b)$. For Turing reducibility, however, this problem does not arise thanks to Spector's theorem \cite[Corollary 5.5]{AshK2000}: if $\ordO{a}=\ordO{b}$, then $H(a) \equiv_{\Tred} H(b)$. 
\begin{definition}
\label{dfn:hyperarith-completeness-hardness}
    A set $A$ is \textbf{$\Sigma^0_\alpha$-complete} (or \textbf{$\Sigma^0_\alpha$-hard}) if $A \equiv_{\Tred} H(a)$ ($A \le_{\Tred} H(a)$ respectively) for some/any $a$ such that $\ordO{a}=\alpha$.    
\end{definition}

\section{Proof theoretic foundations}
\label{sec:pt}
Our main result relies on two fundamental proof theoretic properties of our systems $\muMALLord\alpha$:
\begin{itemize}
    \item For the upper bound, we must show that the restriction of proof search to only non-$\cut$ steps suffices, cf.~\cref{prf-srch-terminates}, by way of a cut-elimination theorem.
   \item For the lower bound we must further rely on a normal form of proofs in order to narrow proof search appropriately, by way of a focussing theorem.
\end{itemize}

Both of these are interesting in their own right, exemplifying the robustness of $\muMALLord\alpha$'s proof theoretic foundations.
This section is devoted to proving both of them.

Let us point out that the soundness and completeness for $\muMALLord\alpha$ for phase semantics in \cite{De2022Thesis} also implies a cut-admissibility result, but this does not imply the convergence of cut-reductions, as we provide here.

\begin{remark}
[Formula ancestry]
Throughout this section we shall employ standard terminology for identifying formula occurrences in inference steps and proofs (see \cite[Section~1.2.3]{Buss1998}).
In particular the \textbf{principal} formula of a logical step is the distinguished formula in its lower sequent, as typeset in \cref{fig:MALL,eq:muord-nuord-rules}.
\textbf{Auxiliary formulas} are any distinguished formulas in any upper sequents.
In an inference step an \textbf{immediate ancestor} of a formula occurrence $A$ in the lower sequent is either an auxiliary formula (if $A$ is principal), or the corresponding occurrence of $A$ in an upper sequent (if $A$ is not principal).
A \textbf{direct ancestor} of a formula occurrence in a proof is a higher occurrence of the same formula connected by a path in the graph of immediate ancestry.\footnote{By the forthcoming notion of formula rank, such a path can never be principal, except perhaps at its upper end point.}
An \textbf{origin} of a formula occurrence is a direct ancestor that is principal.
\end{remark}

\subsection{On termination of cut-free proof search}

$\MALL $, without $\cut$, is the prototypical example of a terminating proof system: bottom-up, continual application of the non-$\cut$ proof rules eventually terminates, as the number of connectives in a sequent is strictly decreasing.
It turns out we can extend this property to our infinitary systems too:

\begin{proposition}
    [Ranks, Theorem~5.2.1 in \cite{De2022Thesis}]
    \label{ranks-exist}
    There is a map $\rank $ from formulas to ordinals s.t.:
    \begin{itemize}
        \item $\rk {A_i} < \rk {A_0 \star A_1}$ for any binary connective $\star$.
        \item $\rk {A(\sigma^\gamma x A(x))} < \rk {\eta^\beta x A(x)} $ for $\eta \in \{\mu,\nu\}$ and $\gamma<\beta$.
    \end{itemize}
\end{proposition}

As we mentioned before we shall ultimately need a notion of rank that tightly bounds the height of proof search for our main complexity results.
To this end we will define in \cref{sec:ranks} a new more refined notion of rank than earlier ones, e.g.\ those from \cite{AlberucciKS14,De2022Thesis}.
Before that, 
let us state a desired consequence that nonetheless suffices for our proof theoretic applications and prerequisite.

\begin{corollary}
\label{prf-srch-terminates}
    Cut-free proof search is terminating.
\end{corollary}
\begin{proof}
[Proof sketch]
    Any inference step, bottom-up, replaces a formula by only ones of smaller rank, thus strictly decreasing the multiset of ranks (see, e.g., \cite[Section~2]{dershowitz1979proving} for an account of multiset orderings).
\end{proof}

Taking the multiset measure above will turn out to be overkill but, as we have already said, \cref{prf-srch-terminates} above already suffices in order to establish our precursory proof theoretic results, cut-elimination and focussing.
Before that, let us see some simpler examples of the proof theoretic behaviour of $\muMALLord \alpha$.

In what follows we shall refer to the \textbf{rank} of a formula $A$ to mean $\rank (A)$, cf.~\cref{ranks-exist}.
We assume only the properties stated in the two results above, \cref{ranks-exist,prf-srch-terminates}, for now.
Note in particular that induction on ranks, for $\MALL$ formulas, subsumes induction on their size.

\begin{example}
    [Functors and $\eta$-expansion, cf.~Theorem~5.2.3 in \cite{De2022Thesis}]
    \label{functors-and-eta}
    One useful application of proof search termination is the justification of our notion of duality.
    For this we shall prove functoriality of formula contexts, namely the admissiblity of the following rules:
    \[
    \vlinf{A(\cdot)}{}{A(B) \seqar A(B')}{B\seqar B'}
    \]
    We proceed by induction on $\rk {A(x)}$.
    The propositional cases are just like in $\MALL$, so the critical cases are the fixed points. For the $\nu$ case we have,
    \[
    \vlinf{\rr \nu}{}{\nuord \beta x A(x,B) \seqar \nuord \beta x A(x,B')}{
    \left\{
    \vlderivation{
    \vlin{\lr \nu}{}{\nuord \beta xA(x,B) \seqar A(\nuord \gamma x A(x,B'),B')}{
    \vlin{\IH}{}{A(\nuord \gamma x A(x,B),B) \seqar A(\nuord \gamma x A(x,B'),B')}{\vlhy{B \seqar B'}}
    }
    }
    \right\}_{\gamma<\beta}
    }
    \]
    where $\IH$ is obtained by the inductive hypothesis.
    Note that the $\mu$-case is identical, by De Morgan duality.
This argument not use $\cut$ and uses only atomic identity steps, further giving a reduction from general identity steps to atomic ones.

This construction is quite different from how analogous results are obtained for $\muMALL$, in both the (co)inductive \cite{Baelde-Miller,baelde2012least} and non-wellfounded \cite{BaeldeDoumaneSaurin2016CSL} settings, where we must strengthen the statement to account for multiple arguments $B$.
\end{example}

\subsection{Cut-elimination}
Let us write $\muMALLordcf \alpha$ for the system $\muMALLord \alpha$ without the $\cut $ rule.
The main result of this subsection is:
\begin{theorem}
    [Cut-elimination]
    \label{cut-elim-mumall}
    $\muMALLord\alpha  \proves \Gamma \implies \muMALLordcf\alpha \proves \Gamma$.
\end{theorem}
The proof is similar to the usual one for $\MALL$, only using induction on ranks rather than usual formula size.
The important new case is that the following key cut,
\[
\vlderivation{
\vliin{\cut}{}{\Gamma, \Delta}{
    \vlin{\mu}{\delta<\beta}{\Gamma, \mu^\beta X A(X) }{\vlhy{\Gamma, A(\mu^\delta X A(X))}}
}{
    \vlin{\nu}{}{\Delta, \nu^\beta \neg A(X)}{\vlhy{\{\Delta, \neg A (\nu^\gamma X \neg A(X))\}_{\gamma < \beta} }}
}
}
\]
is transformed to:
\[
\vliinf{\cut}{}{\Gamma, \Delta}{\Gamma, A(\mu^\delta X A(X))}{\Delta, \neg A(\nu^\delta X \neg A(X))}
\]
Note that the ranks of the premisses of the cut have strictly decreased here, by \cref{ranks-exist}.
Since cut-elimination proofs for $\MALL$, its extensions, and similar terminating systems are prevalent in the literature, we shall employ a rather heavy handed `big step' approach, in order to achieve the result as quickly as possible.

It will be useful to consider the systems $\muMALLordcut \alpha \delta$ allowing cuts only on formulas of bounded rank.
\begin{definition}
    [Bounded cut systems]
    The rank of a $\cut$ step is the rank of its cut formula.
    The system $\muMALLordcut \alpha \delta $ is the restriction of $\muMALLord \alpha $ to only cuts of rank $<\delta$.
\end{definition}

Note that the cut-free system $\muMALLordcf\alpha$ is now just $\muMALLordcut \alpha 0$.
%
%
%
%
Let us establish a standard invertibility lemma for these systems:
\begin{lemma}
    [Invertibility of negatives]
    \label{neg-rules-invertible}
     $\muMALLordcut\alpha \delta$ is closed under the~rules:
    \[
    \vlinf{}{}{\Gamma}{\Gamma, \bot}
    \qquad
    \vlinf{}{}{\Gamma, A, B}{\Gamma, A \vlpa B}
    \qquad
    \vlinf{}{\text{\footnotesize $i \in \{0,1\}$}}{\Gamma, A_i}{\Gamma, A_0 \vlan A_1}
    \qquad
    \vlinf{}{\text{\footnotesize $\gamma<\beta$}}{\Gamma, A(\nuord \gamma x A(x)}{\Gamma, \nuord \beta x A(x)}
    \]
\end{lemma}

Note that we do not need to state any `height-preserving' invariant above, morally as ranks induce a bound also on the height of proofs.

\begin{proof}
Each item is proved in a similar way, simply replacing the principal formula of an invertible step by appropriate auxiliary formulas and deleting appropriate principal steps. Formally:
    \begin{enumerate}
        \item In a $\muMALLordcut \alpha \delta$ proof of $\Gamma,\bot$, delete every direct ancestor of $\bot$. 
        We repair critical steps, at origins of $\bot$, by:
        \[
        \vlinf{\bot}{}{\Sigma, \bot}{\Sigma}
        \quad \leadsto \quad
        \Sigma
        \]
        \item In a $\muMALLordcut \alpha \delta$ proof of $\Gamma,A\vlpa B$, replace every direct ancestor of $A\vlpa B$ by the cedent $A,B$.
        We repair critical steps, at origins of $A\vlpa B$, by:
        \[
        \vlinf{\vlpa}{}{\Sigma , A \vlpa B}{\Sigma, A, B}
        \quad \leadsto \quad
        \Sigma, A,B
        \]
        \item In a $\muMALLordcut \alpha \delta$ proof of $\Gamma,A_0\vlan A_1$, replace every direct ancestor of $A_0\vlan A_1$ by $A_i$. We repair critical steps, at origins of $A_0 \vlan A_i $, by:
        \[
        \vliinf{\vlan}{}{\Sigma, A_0 \vlan A_1}{\Sigma, A_0}{\Sigma, A_1}
        \quad \leadsto \quad
        \Sigma, A_i
        \]
        \item In a $\muMALLordcut \alpha \delta$ proof of $\Gamma,\nuord \beta X A(X)$, replace every direct ancestor of $\nuord \beta X A(X)$ by $A(\nuord \gamma X A(X))$. We repair critical steps, at origins of $\nuord \beta X A(X) $, by:
        \[
        \vlinf{\nuord \beta}{}{\Sigma, \nuord \beta X A(X)}{ \{\Sigma, A(\nuord \gamma X A(X))\}_{\gamma < \beta} }
        \quad \leadsto \quad
        \Sigma, A(\nuord \gamma X A(X))
        \]
    \end{enumerate}
\end{proof}
Now our main technical result is:

\begin{lemma}
\label{d-prfs-closed-under-d-cuts}
    $\muMALLordcut \alpha \delta $ is closed under cuts of rank $\delta$.
\end{lemma}
\begin{proof}
    Suppose in $\muMALLordcut\alpha \delta$ we have proofs $\pi_0$ of $ \Gamma , A$ and $\pi_1$ of $ \Delta, \neg A$, where $A$ has rank $\delta$, and let us show that $\muMALLordcut\alpha \delta \proves \Gamma, \Delta$.\todo{maybe visualise the rule}
    We proceed by case analysis on $A$, reducing to cuts on smaller formulas:
    \begin{itemize}
        \item If $A = 1$ then $\neg A = \bot$. We replace every direct ancestor of $1$ in $\pi_0$ by the cedent $\Delta$, repairing critical cases, at origins of $1$, by,
        \[
        \vlinf{1}{}{1}{}
        \quad \leadsto \quad
        \vltreeder{\inv\pi_1}{\Delta}{\quad}{}{\quad}
        \]
        where $\inv\pi_1$ is obtained from $\pi_1$ by invertibility, \cref{neg-rules-invertible}.
        \item If $A = B_0 \vlor B_1$ then $\neg A = \neg B_0 \vlan \neg B_1$.
        We replace every direct ancestor of $B_0 \vlor B_1$ in $\pi_0$ by the cedent $\Delta$, repairing critical cases, at origins of $B_0 \vlor B_1$, by,
        \[
        \vlinf{\vlor}{}{\Sigma, B_0 \vlor B_1}{\Sigma , B_i}
        \quad \leadsto \quad
        \vlderivation{
        \vliin{\dcut{<\delta}}{}{\Sigma, \Delta}{\vlhy{\Sigma, B_i}}{
            \vltr{\inv\pi_1}{\Delta, \neg B_i}{\vlhy{\quad}}{\vlhy{}}{\vlhy{\quad}}
        }
        }
        \]
        where $\inv\pi_1$ is obtained by invertibility, \cref{neg-rules-invertible}.
        \item If $A = B_0 \vlte B_1$ then $\neg A = \neg B_0 \vlpa \neg B_1$. 
        We replace every direct ancestor of $B_0 \vlte B_1$ in $\pi_0$ by the cedent $\Delta$, repairing critical cases, at origins of $B_0 \vlte B_1$, by,
        \[
        \vliinf{\vlte}{}{\Sigma_0, \Sigma_1 , B_0 \vlte B_1}{\Sigma_0, B_0}{\Sigma_1, B_1}
        \quad \leadsto \quad
        \vlderivation{
        \vliin{\dcut{<\delta}}{}{\Sigma_0, \Sigma_1, \Delta}{
            \vlhy{\Sigma_0, B_0}
        }{
            \vliin{\dcut{<\delta}}{}{\Sigma_1, \Delta, \neg B_0}{
                \vlhy{\Sigma_1, B_1}
            }{
                \vltr{\inv\pi_1}{\Delta, \neg B_0, \neg B_1}{\vlhy{\quad}}{\vlhy{}}{\vlhy{\quad}}
            }
        }
        }
        \]
        where $\inv\pi_1 $ is obtained by invertibility, \cref{neg-rules-invertible}.\todo{can set up $\inv\pi_1$ earlier.}
        \item If $A = \muord \beta X B(X)$ then $\neg A = \nuord \beta X \neg B(X)$.
        We replace every direct ancestor of $\muord \beta X B(X)$ in $\pi_0$ by the cedent $\Delta$, repairing critical cases, at origins of $\muord \beta X B(X)$, by,
        \[
        \vlinf{\muord \beta}{\gamma < \beta}{\Sigma, \muord \beta X B(X)}{\Sigma, B(\muord \gamma X B(X))}
        \quad \leadsto \quad
        \vlderivation{
        \vliin{\dcut{<\delta}}{}{\Sigma, \Delta}{
            \vlhy{\Sigma, B(\muord \gamma X B(X))}
        }{
            \vltr{\inv\pi_1}{\Delta, \neg B (\nuord \gamma X \neg B(X))}{\vlhy{\quad}}{\vlhy{}}{\vlhy{\quad}}
        }
        }
        \]
        where $\inv\pi_1$ is obtained by invertibility, \cref{neg-rules-invertible}.
        \item (All other cases are obtained by duality)
        \qedhere
        \end{itemize}
\end{proof}




\begin{proposition}
    \label{d-cuts-to-cf}
    If $\muMALLordcut \alpha \delta \proves \Gamma$ then $\muMALLordcf \alpha \proves \Gamma$, for all $\delta$.
\end{proposition}
\begin{proof}
    By induction on the ordinal $\delta$.
    Let $\pi$ be a $\muMALLordcut \alpha \delta$ proof of $\Gamma$, and let us proceed by a sub-induction on the structure of $\pi$.
    The construction simply commutes with any non-cut step, always calling the sub-inductive hypotheses.
    So suppose $\pi $ ends with a $\cut$, having form,
      \[
    \vlderivation{
    \vliin{\dcut\epsilon}{}{\Gamma_0, \Gamma_1}{
        \vltr{\pi_0}{\Gamma_0, A}{\vlhy{\quad}}{\vlhy{}}{\vlhy{\quad}}
    }{
        \vltr{\pi_1}{\Gamma_1, \neg A}{\vlhy{\quad}}{\vlhy{}}{\vlhy{\quad}}
    }
    }
    \]
    where $\rk A = \epsilon <\delta$. 
    Applying the sub-inductive hypothesis to the subproofs $\pi_0, \pi_1$, we have that $\muMALLordcf\alpha$ proves the sequents $\Gamma_0, A$ and $\Gamma_1 , \neg A$.
    So a fortiori $\muMALLordcut \alpha \epsilon$ also proves these sequents, whence $\muMALLordcut \alpha \epsilon \proves \Gamma_0,\Gamma_1$, by \cref{d-prfs-closed-under-d-cuts}.
    Finally we conclude by the main inductive hypothesis, as $\epsilon < \delta$.
\end{proof}

We can finally establish our main cut-elimination theorem:

\begin{proof}
    [Proof of \cref{cut-elim-mumall}]
    Let $\pi$ be a $\muMALLord \alpha $ proof of $\Gamma$ and let $\delta$ be the strict supremum of all its cut ranks, so that $\pi$ is in particular a $\muMALLordcut \alpha \delta $ proof.
    Now we simply conclude by \cref{d-cuts-to-cf}.
\end{proof}

\subsection{Focussing}

\emph{Focussing} is a discipline that organises the invertible and non-invertible phases of proof search.
For this we require a further classification of connectives by \textbf{polarity}, determined by whether their introduction rule is invertible (negative) or not (positive):
\begin{center}
\begin{tabular}{c|c|c}
     & positive & negative \\
     \hline
literal & $p$ & $\neg p$ \\
false & $0$ & $\bot$ \\
true & $1$ & $\top$ \\
disjunction & $\vlor$ & $\vlpa $\\
conjunction & $\vlte$ & $\vlan$ \\
fixed point & $\muord \beta x$ & $\nuord \beta  x$
\end{tabular}    
\end{center}
Note that we have, by convention, assigned each propositional symbol $p$ positive and its negation $\neg p$ negative. 
Indeed note that negation flips the positive/negative status of a formulas's main connective.

\begin{figure}
    
\textbf{Identity and focussing rules:}
\[
\vlinf{\identity}{}{p \Downarrow \neg p}{}
\quad
\vlinf{\store}{\text{$A$ positive or literal}}{\Gamma \Uparrow \Delta, A}{\Gamma , A \Uparrow \Delta}
\quad
\vlinf{\decide}{\text{$A$ positive}}{\Gamma, A \Uparrow \cdot }{\Gamma \Downarrow A}
\quad
\vlinf{\release}{\text{$A$ negative or literal}}{\Gamma \Downarrow A}{\Gamma \Uparrow A}
\]
\tikhon{negative or literal $===>$ literal?}
\textbf{Asynchronous phase:}
\[
\vlinf{\top}{}{\Gamma \Uparrow \Delta, \top}{}
\quad
\vlinf{\bot}{}{\Gamma\Uparrow  \Delta, \bot}{\Gamma \Uparrow \Delta}
\quad
\vlinf{\vlpa}{}{\Gamma \Uparrow\Delta, A \vlpa B}{\Gamma \Uparrow \Delta, A, B}
\quad
\vliinf{\vlan}{}{\Gamma \Uparrow \Delta, A \vlan B}{\Gamma \Uparrow \Delta, A}{\Gamma \Uparrow \Delta, B}
\quad
\vlinf{\nuord\beta}{}{\Gamma \Uparrow \Delta, \nu^\beta X A(X)}{
\{
\Gamma \Uparrow \Delta , A(\nu^\gamma X A(X))
\}_{\gamma < \beta}
}
\]

\textbf{Synchronous phase:}
\[
\vlinf{1}{}{\cdot \Downarrow 1}{}
\quad
\vlinf{\vlor}{i\in \{0,1\}}{\Gamma \Downarrow A_0 \vlor A_1}{\Gamma \Downarrow A_i}
\quad
\vliinf{\vlte}{}{\Gamma, \Delta \Downarrow A \vlte B}{\Gamma \Downarrow A}{\Delta \Downarrow B}
\quad
\vlinf{\muord\beta}{\gamma <\beta}{\Gamma \Downarrow \mu^\beta X A(X)}{\Gamma \Downarrow A(\mu^\gamma X A(X))}
\]
    \caption{The cut-free focussed system $\muMALLFord\alpha$.}
    \label{fig:muMALLFord}
\end{figure}

The focussing discipline is usually presented as a decorated version of the underlying proof system, admitting sufficient bookkeeping to enforce the discipline. Sequents of the focussed calculus $\muMALLFord \alpha$ are of the form $\Gamma \Downarrow A$ or $\Gamma \Uparrow \Theta$ where $\Gamma$ is a multiset of formulas, $A$ is a formula and $\Theta$ is a list of formulas.
\begin{definition}
    [Focussed system]
    The system $\muMALLFord \alpha $ is given in \cref{fig:muMALLFord}.
\end{definition}

Notice that, bottom-up, once a positive formula has been chosen principal, an auxiliary formula of each premiss must remain principal as long as it is positive.
This is what significantly narrows the proof search space, further to cut-elimination.
The \emph{focussing theorem} establishes that all proofs can be rewritten so as to fit the focussing discipline. 
It is usually presented as a sort of completeness result:

\begin{theorem}
[Focussing]
\label{focussing-completeness}
    $\muMALLordcf{\alpha} \proves \Gamma \iff \muMALLFord\alpha \proves \cdot \Uparrow \Gamma$.
\end{theorem}

Of course the $\impliedby$ direction is immediate, by simply ignoring focussing decorations.
We shall prove the more difficult $\implies$ direction via rule commutations.
Rather than factoring into several intermediate lemmas, it is quicker to employ a rather `heavy handed' approach and prove the theorem all at once.

\begin{proof}
    We shall prove a slightly stronger result, that if $\muMALLordcf \alpha \proves \Gamma$ then $\muMALLFord\alpha \proves \Gamma^+ \Uparrow \Gamma^-$, where $\Gamma^+$ is the positive part of $\Gamma $ and $\Gamma^-$ is the negative part of $\Gamma$. From here the stated result follows by applying several store rules, $\store$.
    We proceed by induction on the multiset of ranks of formulas in $\Gamma$.

    First, if $\Gamma $ contains a negative formula, we may apply the invertibility lemma and inductive hypotheses to conclude. For instance, if $\pi$ is a $\muMALLordcf \alpha$ proof of $\Delta, \nuord\beta x A(x)$ and $A(x)$ is positive then we construct the $\muMALLFord \alpha$ proof,
    \[
    \vlinf{}{}{\Delta^+ \Uparrow  \Delta^-, \nuord \beta x A(x)}{
    \left\{
    \vlderivation{
    \vlin{\store}{}{\Delta^+ \Uparrow \Delta^-, A(\nuord \gamma x A(x))}{
    \vltr{\IH(\inv\pi)}{\Delta^+, A(\nuord\gamma x A(x)) \Uparrow \Delta^-}{\vlhy{\qquad}}{\vlhy{\quad}}{\vlhy{\qquad}}
    }
    }
    \right\}_{\gamma<\beta}
    }
    \]
    where derivations marked $\IH(\inv \pi)$ are obtained by first applying invertibility to $\pi$, \cref{neg-rules-invertible}, then calling the inductive hypothesis.
    If $A$ is negative we simply omit the uppermost $\store$ steps.
    
    Now suppose $\Gamma $ consists only of positive formulas. Let us conduct a case analysis on the concluding (positive) step of some proof $\pi$ of $\Gamma$ in $\muMALLord\alpha$.

    If $\pi$ has the form,
    \[
    \vlinf{1}{}{1}{}
    \]
    then we construct the required proof as follows:
    \[
    \vlderivation{
    \vlin{\decide}{}{1 \Uparrow\cdot}{
    \vlin{1}{}{\cdot \Downarrow 1}{\vlhy{}}
    }
    }
    \]
    
    If $\pi$ has the form,
        \[
        \vlderivation{
        \vlin{}{}{\Delta, A_0 \vlor A_1}{
        \vltr{\pi_0}{\Delta, A_0}{\vlhy{\quad }}{\vlhy{}}{\vlhy{\quad}}
        }
        }
        \]
        then let $\pi_0'$ be the $\muMALLFord\alpha $ proof obtained from $\pi_0$ by the inductive hypothesis.\todo{explain rank/height compatibility again.}
        We have two subcases:
        \begin{itemize}
            \item If $A_0$ is negative then $\pi_0'$ has conclusion $ \Delta \Uparrow  A_0$. We now construct the required proof as follows:
            \[
            \vlderivation{
            \vlin{\decide}{}{\Delta, A_0 \vlor A_1 \Uparrow \cdot}{
            \vlin{\vlor}{}{\Delta \Downarrow A_0 \vlor A_1}{
            \vlin{\release}{}{\Delta \Downarrow A_0 }{
            \vltr{\pi_0'}{\Delta \Uparrow A_0}{\vlhy{\quad }}{\vlhy{}}{\vlhy{\quad }}
            }
            }
            }
            }
            \]
            \item If $A_0 $ is positive then $\pi_0'$ has conclusion $\Delta, A_0 \Uparrow \cdot$. 
            We replace every direct ancestor of $A_0$ in $\pi_0'$ by $A_0 \vlor A_1$.
        The critical cases, decisions on $A_0$, are repaired as follows:
        \[
        \vlinf{\decide}{}{\Sigma , A_0 \Uparrow \cdot }{\Sigma \Downarrow A_0}
        \quad \leadsto \quad
        \vlderivation{
        \vlin{\decide}{}{\Sigma, A_0 \vlor A_1 \Uparrow \cdot}{
        \vlin{\vlor}{}{\Sigma \Downarrow A_0 \vlor A_1}{
        \vlhy{\Sigma\Downarrow A_0 }
        }
        }
        }
        \]
        \end{itemize}
        
    If $\pi$ finishes with the other $\vlor $ step, with same conclusion and auxiliary formula $A_1$, the reasoning is similar.

    If $\pi $ has the form,
    \[
    \vlderivation{
    \vliin{\vlte}{}{\Gamma_0, \Gamma_1, A_0 \vlte A_1}{
        \vltr{\pi_0}{\Gamma_0 , A_0 }{\vlhy{\quad }}{\vlhy{}}{\vlhy{\quad}}
    }{
        \vltr{\pi_1}{\Gamma_1 , A_1 }{\vlhy{\quad}}{\vlhy{}}{\vlhy{\quad}}
    }
    }
    \]
    then let $\pi_0'$ and $\pi_1'$ be the $\muMALLFord \alpha $ proofs obtained by the inductive hypothesis from $\pi_0$ and $\pi_1$ respectively.
    We have the following subcases:
\begin{itemize}
    \item If $A_0$ and $A_1$ are negative then $\pi_0'$ has conclusion $\Gamma_0\Uparrow A_0$ and $\pi_1'$ has conclusion $\Gamma_1 \Uparrow A_1$.
    We now construct the required proof as follows:
    \[
    \vlderivation{
    \vlin{\decide}{}{\Gamma_0, \Gamma_1 , A_0 \vlte A_1 \Uparrow \cdot}{
    \vliin{\vlte}{}{\Gamma_0, \Gamma_1 \Downarrow A_0 \vlte A_1 }{
        \vlin{\release}{}{\Gamma_0\Downarrow A_0}{
        \vltr{\pi_0'}{\Gamma_0 \Uparrow A_0}{\vlhy{\quad}}{\vlhy{}}{\vlhy{\quad}}
        }
    }{
        \vlin{\release}{}{\Gamma_1 \Downarrow A_1}{
        \vltr{\pi_1'}{\Gamma_1 \Uparrow A_1}{\vlhy{\quad}}{\vlhy{}}{\vlhy{\quad}}
        }
    }
    }
    }
    \]
    \item Otherwise suppose, WLoG, $A_0$ is positive, so that $\pi_0'$ has conclusion $\Gamma_0 , A_0 \Uparrow\cdot$.
    We replace every direct ancestor of $A_0$ in $\pi_0'$ by the cedent $\Gamma_1 , A_0 \vlte A_1$.
    The critical cases are decisions on $A_0$ in $\pi_0'$:
    \begin{equation}
        \label{eq:crit-case-tensor-replacement}
        \vlinf{\decide}{}{\Sigma, A_0 \Uparrow \cdot}{\Sigma \Downarrow A_0}
        \quad \leadsto \quad
        \vlinf{}{}{\Sigma, \Gamma_1, A_0 \vlte A_1 \Uparrow\cdot}{?}
    \end{equation}
    These are repaired according to a subsubcase analysis on $A_1$:
    \begin{itemize}
        \item If $A_1$ is negative then $\pi_1'$ has conclusion $\Gamma_1 \Uparrow A_1$ and we repair the critical case \cref{eq:crit-case-tensor-replacement} by:
        \[
        \vlderivation{
        \vlin{\decide}{}{\Sigma, \Gamma_1 , A_0 \vlte A_1 \Uparrow\cdot}{
        \vliin{\vlte}{}{\Sigma, \Gamma_1 \Downarrow A_0 \vlte A_1}{
            \vlhy{\Sigma \Downarrow A_0}
        }{
            \vlin{\release}{}{\Gamma_1 \Downarrow A_1}{
            \vltr{\pi_1'}{\Gamma_1 \Uparrow A_1}{\vlhy{\quad}}{\vlhy{}}{\vlhy{\quad}}
            }
        }
        }
        }
        \]
        \item If $A_1$ is positive then $\pi_1'$ has conclusion $\Gamma_1, A_1 \Uparrow\cdot$.
        We repair the critical case \cref{eq:crit-case-tensor-replacement} by replacing every direct ancestor of $A_1 $ in $\pi_1'$ by $\Sigma, A_0 \vlte A_1$. 
        Note that this indeed produces the correct conclusion, cf.~\cref{eq:crit-case-tensor-replacement}.
        Now the new critical cases are  decisions on $A_1$ in $\pi_1'$, which we repair as follows:
        \[
        \vlinf{\decide}{}{\Pi, A_1 \Uparrow \cdot}{\Pi \Downarrow A_1}
        \quad \leadsto \quad
        \vlderivation{
        \vlin{\decide}{}{\Pi, \Sigma, A_0 \vlte A_1 \Uparrow\cdot}{
        \vliin{\vlte}{}{\Pi, \Sigma \Downarrow A_0 \vlte A_1}{
            \vlhy{\Sigma\Downarrow A_0}
        }{
            \vlhy{\Pi \Downarrow A_1}
        }
        }
        }
        \]
    \end{itemize}
\end{itemize}

If $\pi$ has the form,
\[
\vlderivation{
\vlin{\muord \beta}{\gamma< \beta}{\Delta, \muord \beta X A(X)}{
\vltr{\pi_0}{\Delta, A(\muord \gamma X A(X))}{\vlhy{\quad}}{\vlhy{}}{\vlhy{\quad}}
}
}
\]
then let $\pi_0'$ be the $\muMALLFord\alpha$ proof obtained from $\pi_0$ by the inductive hypothesis. We have two subcases:
\begin{itemize}
    \item If $A(\muord \gamma X A(X))$ is negative then $\pi_0'$ has conclusion $\Delta \Uparrow A(\muord \gamma X A(X))$. 
    We now construct the required proof as follows:
    \[
    \vlderivation{
    \vlin{\decide}{}{\Delta,\muord \beta X A(X) \Uparrow \cdot}{
    \vlin{\muord \beta}{}{\Delta \Downarrow \muord \beta X A(X)}{
    \vlin{\release}{}{\Delta \Downarrow A(\muord \gamma X A(X))}{
    \vltr{\pi_0'}{\Delta \Uparrow A(\muord \gamma X A(X))}{\vlhy{\quad}}{\vlhy{}}{\vlhy{\quad}}
    }
    }
    }
    }
    \]
    \item If $A(\muord \gamma X A(X))$ is positive then $\pi_0'$ has conclusion $\Delta, A(\muord \gamma X A(X)) \Uparrow \cdot$.
    We replace every direct ancestor of $A(\muord \gamma X A(X))$ in $\pi_0'$ by $\muord \beta X A(X)$. 
    The critical cases, decisions on $A(\muord \gamma X A(X))$, are repaired as follows:
    \[
    \vlinf{\decide}{}{\Sigma, A(\muord \gamma X A(X)) \Uparrow\cdot}{\Sigma \Downarrow A(\muord \gamma X A(X))}
    \quad \leadsto \quad
    \vlderivation{
    \vlin{\decide}{}{\Sigma, \muord \beta X A(X) \Uparrow \cdot}{
    \vlin{\muord \beta}{}{\Sigma \Downarrow \muord \beta X A(X)}{
    \vlhy{\Sigma \Downarrow A(\muord \gamma X A(X))}
    }
    }
    }
    \qedhere
    \]
\end{itemize}
\end{proof}

\section{Bounds on formula ranks}
\label{sec:ranks}

In this section we shall make explicit the rank function $\rank$ mentioned earlier and duly prove \cref{ranks-exist}.
Our intention is to define ranks in a \emph{tight} manner, so that we can readily obtain our lower and upper bounds on provability later. 
This will also induce a more fine-grained notion of rank of a sequent than the multiset measure of \cref{prf-srch-terminates}.
For this reason we shall also carry out a number of calculations for bounding formula ranks.

Henceforth we shall write $\Ord$ for the class of all ordinals. 

\subsection{Natural operations on ordinals}
Due to commutativity of the binary connectives, our definition of rank will rather work with commutative variants of arithmetic operations on $\Ord$, known as the \emph{natural} operations.
In this subsection we shall recall some of their properties.
We assume the reader's familiarity with basic ordinal arithmetic, namely, with ordinal addition, multiplication and exponentiation. We denote a sum of the form $\alpha_0+\ldots+\alpha_{k-1}$ by $\sum\limits_{i<k} \alpha_i$.

\begin{definition}
[Degrees]
    The \textbf{degree} $\deg \alpha$ of an ordinal $\alpha>0$ is the unique ordinal $\beta$ such that $\alpha = \omega^{\beta} \cdot n + \gamma$ for some $0 < n < \omega$ and some $\gamma<\omega^\beta$.
\end{definition}

\begin{proposition}\label{proposition:omega-power}
    If $\alpha \ge \omega$, then $\alpha^\omega = \omega^{(\deg \alpha) \omega}$.
\end{proposition}
\begin{proof}
    $(\omega^{\deg \alpha})^\omega \le \alpha^\omega \le \left(\omega^{\deg \alpha+1}\right)^\omega = \omega^{(\deg \alpha+1)\omega} = \omega^{(\deg \alpha) \omega}$. The last inequality holds since $\deg \alpha > 0$.
\end{proof}

\begin{definition}
[Natural sum and product]
\label{definition:natural-sum}
    Write 
    $\alpha = \sum\limits_{i<k} \omega^{\gamma_i} a_i$ and 
    $\beta = \sum\limits_{i<k} \omega^{\gamma_i} b_i$
    where $\gamma_0 > \ldots > \gamma_{k-1}$ and $a_i,b_i < \omega$. The \textbf{natural sum} $\alpha \natsum \beta$ is defined as:
    \[
    \alpha \natsum \beta \df \sum\limits_{i<k} \omega^{\gamma_i} (a_i + b_i)
    \]
    and the \textbf{natural product} $\alpha \natprod \beta$ is defined as:
    \[
    \alpha \natprod \beta \df \bignatsum_{i, j < k} {(\omega^{\gamma_i \natsum \gamma_j} a_ib_j)}.
    \]
\end{definition}
Here $\bignatsum_{i \in I} \alpha_i$ is $0$ if $I = \emptyset$, else $\alpha_j \natsum\bignatsum_{i \in I \setminus\{j\}} \alpha_i$ for some $j \in I$ (the index set $I$ is finite). This is a well-defined operation independent of the choice of $j$, thanks to the following proposition.

\begin{proposition}[{\cite[Lemma 3.3]{deJonghP77}}]\label{proposition:properties-natural-arithmetic}
    \leavevmode
    \begin{enumerate}
        \item $\natsum$ and $\natprod$ are commutative, associative and monotone operations.
        \item $\natprod$ distributes over $\natsum$: $(\alpha \natsum \beta) \natprod (\alpha \natsum \gamma) = \alpha \natprod (\beta \natsum \gamma)$.
        \item\label{item:natsum-omega^gamma} If $\alpha < \omega^\gamma$ and $\beta < \omega^\gamma$, then $\alpha \natsum \beta < \omega^\gamma$.
    \end{enumerate}
\end{proposition}

Note that we are using unconventional notations $\natsum$, $\natprod$ for natural arithmetic operations so as not to clash with the linear logic operations $\oplus$, $\otimes$. 

\begin{definition}
    Let us define \textbf{iterated natural product} $\alpha^{\natprod \beta}$ as follows:
    \begin{align*}
        \alpha^{\natprod 0} &\df 1;
        \\
        \alpha^{\natprod \beta} & \df \bigcup\limits_{\gamma < \beta} (\alpha^{\natprod \gamma} \natprod \alpha)
    \end{align*}
\end{definition}

\begin{proposition}[{\cite[Lemma 3.8]{deJonghP77}}]\label{proposition:iterated-natural-product}
    $(\omega^{\omega^\alpha})^{\natprod \beta} = (\omega^{\omega^\alpha})^\beta$ for any $\alpha,\beta \in \Ord$.
\end{proposition}

We do not call the above operation natural exponentiation in order not to confuse with the notion of natural exponentiation from \cite{Altman17}. There, natural exponentiation $e(\alpha,\beta)$ is defined as an operation non-decreasing w.r.t.~$\alpha$ and, for $\alpha>0$, non-decreasing w.r.t.~$\beta$, satisfying the laws:
\begin{enumerate}
    \item $e(\alpha,1)=\alpha$, 
    \item $e(\alpha,\beta \natsum \gamma) = e(\alpha,\beta) \natprod e(\alpha,\gamma)$, 
    \item $e(\alpha,\beta \natprod \gamma) = e(e(\alpha,\beta), \gamma)$.
\end{enumerate}
Iterated natural product, called \emph{super-Jacobsthal exponentiation} in \cite{Altman17}, does not satisfy these laws. In fact, \cite{Altman17} shows that natural exponentiation, as determined by the laws above, \textit{does not exist}.

\subsection{Definition of formula ranks}
We are now finally ready to present our definition of formula rank, cf.~\cref{ranks-exist}.

\begin{definition}
[Partial valuations]
    A \textbf{partial valuation} is a partial function $s : \Var \parto \Ord$.
    \begin{itemize}
        \item If $\dom s = \{x\}$ where $x \in \Var$, then we denote $s$ as $x \mapsto s(x)$. 
        \item We write $1$ to denote the valuation $s$ such that $s(x)=1$ for each $x \in \Var$.
        \item We write $s \ge s^\prime$ for two partial valuations if $s(x)\ge s^\prime(x)$ for each $x \in \dom s \cap \dom s^\prime$.
        \item If $s_1$ and $s_2$ are two partial valuations, then `$s_1,s_2$' denotes the following partial valuation $s$:
        \[
        s(x) \df 
        \begin{cases}
            s_2(x), & x \in \dom s_2
            \\
            s_1(x), & x \in \dom s_1 \setminus \dom s_2
            \\
            \text{undefined}, & \text{otherwise}
        \end{cases}
        \]
        \item If $s$ is a partial valuation and $\gamma$ is an ordinal, then $\gamma \natprod s$ is a partial valuation $s^\prime$ defined as follows:
        \[
        s^\prime(x) \df
        \begin{cases}
            \gamma \natprod s(x), & x \in \dom s
            \\
            \text{undefined}, & \text{otherwise}
        \end{cases}
        \]
    \end{itemize}
\end{definition}

\begin{definition}
[Ranks]
\label{definition:rank}
    Let $s:\Var \to \Ord$ be a total valuation function mapping variables to ordinals. 
    We define $\valuation{A}{s}$ as follows:
    \begin{itemize}
        \item $\valuation{c}{s} \df 1$ if $c$ has form $p,\neg p, 0,\top, 1, \bot$;
        \item $\valuation{x}{s} \df s(x)$ for $x\in \Var$;
        \item $\valuation{A_1 \circ A_2}{s} \df 1 \natsum \valuation{A_1}{s} \natsum \valuation{A_2}{s}$ for $\circ \in \{\vlpa, \vlor,\vlte,\vlan\}$;
        \item $\valuation{\eta^\beta x.B}{s} \df  \sup\limits_{\gamma<\beta}\left\{ \valuation{B}{s,x\mapsto \valuation{\eta^\gamma x.B}{s}} + 1 \right\}$ for $\eta \in \{\mu,\nu\}$.
    \end{itemize}
    Define $\rank(A) \df \valuation{A}{1}$. If $\Gamma = A_1,\ldots,A_n$ is a multiset of formulae, then $\rank(\Gamma) \df \rank(A_1)\natsum \ldots \natsum \rank(A_n)$.
\end{definition}

Note that, unlike the earlier proof of \cref{prf-srch-terminates}, we here calculate the rank of a sequent directly as an ordinal by taking natural sums of its formulas. 
The reason for this, again, is so that we may assign an explicit ordinal measure to each sequent, towards proving our main result, \cref{theorem:complexity-main}.

A similar rank is introduced for modal $\mu$-formulae in \cite{AlberucciKS14} in order to justify induction on such formulae. However, there is one significant difference between the rank introduced above and that from \cite{AlberucciKS14}: the latter's definition for the binary operation case involves maximum instead of natural sum, i.e. $\valuation{A_1 \circ A_2}{s}$ is defined as $\max\{\valuation{A_1}{s} , \valuation{A_2}{s}\}+1$ in \cite{AlberucciKS14}. The reason why \cref{definition:rank} involves natural sum instead is because we want each inference step of $\muMALLordcf{\alpha}$ to strictly decrease the rank of a sequent, bottom-up.
Note that, if ranks of formulae and sequents were defined using maximum, this result would not hold. Indeed, in this case, the rule for multiplicative disjunction `$\vlpa$' would have a premise and a conclusion of the same rank.

\subsection{Upper bound on ranks}
In this subsection we compute appropriate upper bounds on our notion of rank.
As a result we are finally able to establish  \cref{ranks-exist}.

\begin{lemma}
[Upper bound]
\label{lemma:ranks-upper-bound}
    We have the following:
    \begin{enumerate}
        \item\label{lemma:rub(1)} If $s\ge 1$, then $\valuation{A}{s} \ge 1$; if $s \ge s^\prime$, then $\valuation{A}{s} \ge \valuation{A}{s^\prime}$.
        \item\label{lemma:rub(2)} $\valuation{A[B/x]}{s} = \valuation{A}{s,x\mapsto \valuation{B}{s}}$.
        \item\label{lemma:rub(3)} If $s \ge 1$, $s^\prime \ge 1$ and $\gamma \ge 1$, then $\valuation{A}{s,\gamma \natprod s^\prime} \le \gamma \natprod \valuation{A}{s,s^\prime}$.
        \item\label{lemma:rub(4)} $\rank(\eta^\beta y. C) \le (\rank(C) + 1)^{\natprod \beta}$.
        \item\label{lemma:rub(5)} $\rank(A) < \omega^{\alpha^\omega}$.
    \end{enumerate}
\end{lemma}

Before giving the proof let us show how this result can be used to 
establish \cref{ranks-exist}.
In fact we shall prove a stronger statement, subsuming also \cref{prf-srch-terminates}:

\begin{proposition}\label{proposition:ranks}
 For each inference step of $\muMALLordcf \alpha$, the rank of its conclusion is greater than any of the ranks of its premises.
\end{proposition}

Note that this indeed yields in particular \cref{ranks-exist} by identifying $\rk A$ with the rank of the singleton sequent $A$.

\begin{proof}
    The only nontrivial part concerns the rules
    \[
    \vlinf{\muord \beta }{\gamma <\beta}{\Gamma, \muord \beta x. B}{\Gamma, B[\muord \gamma x. B/x]}
    \qquad
    \vlinf{\nuord \beta}{}{\Gamma, \nuord \beta x . B}{ \{ \Gamma, B[\nuord \gamma x . B/x] \}_{\gamma < \beta}  }
    \]
    The following computation (for $\eta \in \{\mu,\nu\}$ and $\gamma<\beta$) proves the statement:
    \begin{align*}
        \rank(\Gamma, \eta^\beta x. B) 
        & = \rank(\Gamma) \natsum \rank(\eta^\beta x. B) 
        \\
        & > \rank(\Gamma) \natsum \valuation{B}{1, x \mapsto \rank(\eta^\gamma x. B)} 
        && \text{\cref{definition:rank}}
        \\
        & = \rank(\Gamma) \natsum \rank(B[\eta^\gamma x. B/x]) 
        && \text{\cref{lemma:ranks-upper-bound}(2)}
        \\
        & = \rank(\Gamma, B[\eta^\gamma x. B/x])
    \end{align*}
    (Note that the left-hand side of each of the relations `$=$' and `$>$' above is the one immediately above the right-hand side.) 
\end{proof}

Now let us prove the Lemma as required:

\begin{proof}
[Proof of \cref{lemma:ranks-upper-bound}]
    \cref{lemma:rub(1)} is proved by a straightforward induction.
    
    \cref{lemma:rub(2)} is proved by double induction on the size of $A$ and on $\beta$ if $A = \eta^\beta y. C$. Let us consider the latter case, which is the only nontrivial one. We assume that $y \notin \FVar(B)$. Let $s^\prime(x) \df \valuation{B}{s}$ and $s^\prime(z) \df s(z)$ for $z \ne x$, $z \in \Var$.
    \begin{align*}
        \valuation{\eta^\beta y. C[B/x]}{s} 
        & = 
        \sup\limits_{\gamma<\beta}\left\{\valuation{C[B/x]}{s,y\mapsto \valuation{\eta^\gamma y. C[B/x]}{s}} +1 \right\}
        && \text{\cref{definition:rank}}
        \\
        & = 
        \sup\limits_{\gamma<\beta}\left\{\valuation{C}{s^\prime, y\mapsto \valuation{\eta^\gamma y. C[B/x]}{s}} +1 \right\}
        && \text{I.H.}
        \\
        & = 
        \sup\limits_{\gamma<\beta}\left\{\valuation{C}{s^\prime, y\mapsto \valuation{\eta^\gamma y. C}{s^\prime}} +1 \right\}
        && \text{I.H.}
        \\
        & =
        \valuation{\eta^\beta y. C}{s^\prime}
    \end{align*}
    \cref{lemma:rub(3)} is proved by induction on the same parameters. If $A$ is a logical constant or a literal, then $\valuation{A}{s, \gamma \natprod s^\prime} = 1 \le \gamma = \gamma \natprod \valuation{A}{s, s^\prime}$. If $A \notin \dom s^\prime$ is a variable, then $\valuation{A}{s,\gamma \natprod s^\prime} = \valuation{A}{s,s^\prime} \le \gamma \natprod \valuation{A}{s,s^\prime}$. If $A \in \dom s^\prime$, then $\valuation{A}{s, \gamma \natprod s^\prime} = \gamma \natprod s^\prime(A) = \gamma \natprod \valuation{A}{s, s^\prime}$.
    
    If $A = A_1 \circ A_2$, then 
    \begin{align*}
        \valuation{A_1 \circ A_2}{s,\gamma \natprod s^\prime}
        & =
        1 \natsum
        \valuation{A_1}{s,\gamma \natprod s^\prime} \natsum \valuation{A_2}{s,\gamma \natprod s^\prime}
        && \text{\cref{definition:rank}}
        \\
        & \le
        \gamma \natsum
        \valuation{A_1}{s,\gamma \natprod s^\prime} \natsum \valuation{A_2}{s,\gamma \natprod s^\prime}
        && \gamma \ge 1
        \\
        & \le 
        (\gamma \natprod 1) \natsum
        (\gamma \natprod \valuation{A_1}{s,s^\prime}) \natsum (\gamma \natprod \valuation{A_2}{s,s^\prime})
        &&
        \text{I.H.}
        \\
        &
        =
        \gamma \natprod (1 \natsum
        \valuation{A_1}{s,s^\prime} \natsum \valuation{A_2}{s,s^\prime})
        &&
        \text{\cref{proposition:properties-natural-arithmetic}}
        \\
        &
        =
        \gamma \natprod \valuation{A_1 \circ A_2}{s,s^\prime}
        &&
    \end{align*}
    
    Finally, let $A = \eta^\beta z. C$. Then
    \begin{align*}
        \valuation{\eta^\beta z. C}{s,\gamma \natprod s^\prime}
        & =
        \sup\limits_{\delta<\beta}\left\{\valuation{C}{s,\gamma \natprod s^\prime, z \mapsto \valuation{\eta^\delta z. C}{s,\gamma \natprod s^\prime}} +1 \right\}
        &&
        \\
        & \le
        \sup\limits_{\delta<\beta}\left\{\valuation{C}{s,\gamma \natprod s^\prime, z \mapsto \gamma \natprod \valuation{\eta^\delta z. C}{s, s^\prime}} +1 \right\}
        &&
        \text{I.H.}
        \\
        & \le
        \sup\limits_{\delta<\beta}\left\{ \gamma \natprod \valuation{C}{s,s^\prime, z \mapsto \valuation{\eta^\delta z. C}{s,s^\prime}} +1 \right\}
        &&
        \text{I.H.}
        \\
        & \le 
        \sup\limits_{\delta<\beta} \left\{ \gamma \natprod \left(\valuation{C}{s,s^\prime, z \mapsto \valuation{\eta^\delta z. C}{s,s^\prime}} +1 \right)\right\}
        && \gamma \ge 1
        \\
        & \le
        \gamma \natprod \sup\limits_{\delta<\beta}\left\{ \valuation{C}{s,s^\prime, z \mapsto \valuation{\eta^\delta z. C}{s,s^\prime}} +1 \right\}
        &&
        \\ & =
        \gamma \natprod \valuation{\eta^\beta z. C}{s, s^\prime}
    \end{align*}
    \cref{lemma:rub(4)} is proved by induction on $\beta$. For the base case, $\beta=0$, we have $\rank(\eta^0 y. C) = 1 = (\rank(C)+1)^{\natprod 0 }$.

    The induction step is proved as follows:
    \begin{align*}
        \rank(\eta^\beta y. C)
        & =
        \sup\limits_{\gamma<\beta}\left\{ \valuation{C}{1,y\mapsto \valuation{\eta^\gamma y.C}{1}} + 1 \right\}
        \\
        & \le 
        \sup\limits_{\gamma<\beta}\left\{ \valuation{\eta^\gamma y.C}{1} \natprod \valuation{C}{1} + 1 \right\}
        &&
        \text{\cref{lemma:rub(3)}}
        \\
        & \le
        \sup\limits_{\gamma<\beta}\left\{ (\rank(C) + 1)^{\natprod \gamma} \natprod \rank(C) + 1 \right\}
        &&
        \text{I.H.}
        \\
        & \le
        \sup\limits_{\gamma<\beta}\left\{ (\rank(C) + 1)^{\natprod \gamma} \natprod (\rank(C) + 1) \right\}
        \\
        & = 
        \sup\limits_{\gamma<\beta}\left\{ (\rank(C) + 1)^{\natprod \gamma+1} \right\}
        \\
        & = 
        (\rank(C) + 1)^{\natprod \beta}
    \end{align*}

    Finally, \cref{lemma:rub(5)} is proved by induction on the size of $A$. The cases where $A$ is a variable, a literal, or a logical constant are trivial since $1 < \omega^{\alpha^\omega}$. The case $A  = A_1 \circ A_2$ follows from \cref{item:natsum-omega^gamma} of \cref{proposition:properties-natural-arithmetic}. 
    Let us consider the case $A = \eta^\beta z. C$. By the induction hypothesis, $\rank(C) < \omega^{\alpha^\omega} = \omega^{\omega^{(\deg \alpha) \omega}}$, hence $\rank(C) < \omega^{\omega^{(\deg \alpha) m}}$ for some $m < \omega$. Now, the reasoning proceeds as follows:
    \begin{align*}
        \rank(\eta^\beta z. C) 
        & \le 
        (\rank(C)+1)^{\natprod \beta}
        && \text{\cref{lemma:rub(4)}} 
        \\
        & \le
        (\rank(C)+1)^{\natprod \alpha}
        && \beta \le \alpha
        \\
        & \le
        \left(\omega^{\omega^{(\deg \alpha) m}}\right)^{\natprod \alpha}
        \\
        & =
        \omega^{\omega^{(\deg \alpha) m}\alpha}
        &&
        \text{\cref{proposition:iterated-natural-product}}
        \\
        & \le
        \omega^{\omega^{(\deg \alpha) m}\omega^{\deg \alpha + 1}}
        \\
        & \le
        \omega^{\omega^{(\deg \alpha) (m+2)}} < \omega^{\omega^{(\deg \alpha)\omega}} = \omega^{\alpha^\omega} \qedhere
    \end{align*}
\end{proof}

\subsection{Lower bound on ranks}

Now, let us establish lower bounds on ranks. It turns out that the bound from \cref{lemma:ranks-upper-bound} is tight. To show this, let us fix a sequence of variables $(x_i)_{i<\omega}$ and define formulae $\Resource_n$ as follows.
\begin{definition}\label{definition:rho-formula}
    Define the formulas $R_n $, for $n<\omega$ inductively by:
    \[
    \begin{array}{r@{\ \df \ }l}
        \Resource_0  & x_0 \\
        \Resource_{n+1} & \mu^\alpha x_{n}.\left(x_{n+1} \vlpa \Resource_n\right)
    \end{array}
    \]
    %
\end{definition}

\begin{lemma}\label{lemma:rho-ranks-lower-bound}
    $\valuation{\Resource_n}{x_n\mapsto \gamma} \ge \gamma \cdot f(n)$ where $f(0)=1$ and $f(n+1) = \omega^{\omega^{(\deg \alpha) n}}$.
\end{lemma}

\begin{proof}
    By induction on $n$. The case $n=0$ is trivial, so let us prove the induction step right away. Observe that,
    \[
        \valuation{\Resource_{n+1}}{x_{n+1}\mapsto \gamma} 
         = \valuation{\muord \alpha x_{n}.\left(x_{n+1} \vlpa \Resource_n\right)}{x_{n+1}\mapsto \gamma}
         = \sup\limits_{\beta<\alpha} (o_\beta+1)
    \]
    where,
    \begin{align*}
        o_\beta &= \valuation{x_{n+1} \vlpa \Resource_n}{x_{n+1}\mapsto \gamma,x_n \mapsto p_\beta} \\
        p_\beta &= \valuation{\muord \beta x_{n}.\left(x_{n+1} \vlpa \Resource_n\right)}{x_{n+1}\mapsto \gamma}
    \end{align*}
    Let us establish lower bounds on $o_\beta$. First, $o_1 \ge \valuation{x_{n+1}}{x_{n+1}\mapsto \gamma} = \gamma$. Secondly,
    \begin{align*}
        o_{\beta+1} 
        &= \valuation{x_{n+1} \vlpa \Resource_n}{x_{n+1}\mapsto \gamma, x_n \mapsto p_{\beta+1}} \\
        & = 1 \natsum \gamma \natsum \valuation{\Resource_n}{x_n \mapsto p_{\beta+1}} \\
        & \ge \gamma \natsum (p_{\beta+1} \cdot f(n)); 
        && \text{I.H.} \\
        p_{\beta+1} 
        & = \valuation{\muord {\beta+1} x_{n}.\left(x_{n+1} \vlpa \Resource_n\right)}{x_{n+1}\mapsto \gamma} \\
        & \ge o_\beta + 1.
    \end{align*}
     Combining the inequalities, we obtain that $o_{\beta+1} \ge \gamma \natsum (o_\beta \cdot f(n))$.
    Finally, notice that $o_\beta \ge \sup\limits_{\beta^\prime < \beta} o_{\beta^\prime}$ for limit $\beta$. Using all this, we deduce the following.
    \begin{itemize}
        \item If $n=0$, then $o_{\beta+1} \ge \gamma \natsum o_\beta \ge  o_\beta + \gamma$, hence $o_\beta \ge \gamma \cdot \beta$ and
        \[
        \valuation{\Resource_{n+1}}{x_{n+1}\mapsto \gamma} = \sup\limits_{\beta < \alpha} (o_\beta+1) \ge \sup\limits_{\beta < \alpha} \gamma \cdot \beta = \gamma \cdot \alpha \ge \gamma \cdot \omega = \gamma \cdot f(1).
        \]
        
        \item If $n>0$, then $o_{\beta+1} \ge \gamma \natsum (o_\beta \cdot f(n)) \ge o_\beta \cdot f(n)$, which implies that $o_\beta \ge \gamma \cdot f(n)^{\beta-1}$ for $1 \le \beta < \omega$ and $o_\beta \ge \gamma \cdot f(n)^{\beta}$ for $\beta \ge \omega$. Therefore,
        \begin{align*}
            \valuation{\Resource_{n+1}}{x_{n+1}\mapsto \gamma} 
            & = \sup\limits_{\beta < \alpha} (o_\beta+1)
            \\
            & \ge \sup\limits_{\omega \le \beta < \alpha} \gamma \cdot f(n)^{\beta} \\ 
            & = \sup\limits_{\omega \le \beta < \alpha} \gamma \cdot \omega^{\omega^{(\deg \alpha) (n-1)}\beta}
            \\
            & \ge \sup\limits_{\omega \le \beta < \omega^{\deg \alpha}} \gamma \cdot \omega^{\omega^{(\deg \alpha) (n-1)}\beta}
            \\
            & = \gamma \cdot \omega^{\omega^{(\deg \alpha) n}} = \gamma \cdot f(n+1)
            && \text{$\omega^{\deg \alpha}$ is limit} \qedhere
        \end{align*}
    \end{itemize}
\end{proof}

\begin{corollary}
    $\sup\limits_{n<\omega} \rank(\Resource_n) = \omega^{\alpha^\omega}$.
\end{corollary}
\begin{proof}
    $\sup\limits_{n<\omega} \rank(\Resource_n) \ge \sup\limits_{n<\omega} f(n) = \sup\limits_{n<\omega} \omega^{\omega^{(\deg \alpha) (n-1)}} = \omega^{\omega^{(\deg \alpha) \omega}} = \omega^{\alpha^\omega}$. The other inequality follows trivially from \cref{lemma:ranks-upper-bound}.
\end{proof}

\section{Upper bounds on provability in $\muMALLord \alpha$}
\label{sec:upper-bd}

The proof of the upper bound for Theorem \ref{theorem:complexity-main} relies on a general fact about inductive definitions proved e.g.~in \cite[Theorem 6.2]{KuznetsovS22}, which we present below. 
\begin{definition}
[Positive formulas]
    Let $\Phi(\UniNumA,\UniSetA)$ be an arithmetical formula with a free variable $\UniNumA$ and a fresh unary predicate $\UniSetA$.
    $\Phi(\UniNumA,\UniSetA)$ is \textbf{positive} if $\UniSetA$ does not occur under the scope of a negation. 
\end{definition}
We write $\mathbb{N} \vDash \varphi$ if a formula $\varphi$ is true in the standard model of arithmetic. For $S \subseteq \mathbb N$ define:
\[
[\Phi](S) \ \df \ \{n \in \omega \mid \mathbb{N} \vDash \Phi(n,S) \}
\qquad
[\Phi]^{\beta}
\ \df \ 
\bigcup\limits_{\gamma<\beta} [\Phi]([\Phi]^{\gamma})
\]

\begin{theorem}[\cite{KuznetsovS22}]\label{theorem:reducibility-inductive-definitions}
    Let $b \in \KleeneO$ be a notation of a limit ordinal $\beta = \ordO{b}$. Then, $[\Phi]^{\beta}$ is computably enumerable in $H(b)$.
\end{theorem}
At the level of informal intuition, $[\Phi]^{\beta}$ can be thought of as a set defined by the informal predicate: 
\[
\underbrace{
\Phi(\cdot,\{n \mid \Phi(n,\{n \mid \Phi(n,\ldots \{n \mid \Phi(n,\emptyset)\} \ldots )\})\})
}_{\text{``$\beta$ times''}}
\]
If $\Phi$ in prenex normal form has $k$ quantifier alternations, then the above informal predicate has ``$k \cdot \beta$ quantifier alternations,'' and since $\beta$ is limit, $k \cdot \beta = \beta$. 
Again informally, this gives us a set at the $\beta$-th level of hyperarithmetical hierarchy. This is of course nothing but a high-level metaphor behind the proof of \cref{theorem:reducibility-inductive-definitions}, which is duly formalized in \cite{KuznetsovS22}. \anupam{the remainder of this section could probably be only briefly summarised for the conference submission.}

Fix some G\"odel numbering of $\muMALLord \alpha$ sequents, mapping a sequent $\Gamma$ to its G\"odel number $\sharp \Gamma \in \omega$. We need to assign G\"odel numbers, in particular, to all formulae of the form $\muord \beta x C$ for $\beta < \alpha$. This is done by fixing $a \in \KleeneO$ such that $\ordO{a} = \alpha$ and identifying $\beta$ with the unique $b \lessO a$ such that $\ordO{b} = \beta$. 

Let $\mathcal{D}^\prime(\UniNumA,\UniSetA)$ be the predicate `$\UniNumA = \sharp \Gamma$ for some $\Gamma$ and $\Gamma$ is the conclusion of a $\muMALLordcf \alpha$ step such that, for each its premisses $\Delta$, already $\sharp \Delta \in \UniSetA$.' 
It is routine to show that $\mathcal{D}^\prime(\UniNumA,\UniSetA)$ can be expressed by a positive formula. For example, given the rule
\[
\vlinf{\nuord \beta}{}{\Gamma, \nuord \beta x . B}{ \{ \Gamma, B[\nuord \gamma x . B/x] \}_{\gamma < \beta}  }
\]
consider the sets
\begin{align*}
    P & \df \{(b, \sharp(\Gamma, \nuord \beta x . B)) \mid \ordO{b} = \beta, \ b \le_{\KleeneO} a \} \\
    Q & \df \{(b, c, \sharp(\Gamma, \nuord \beta x . B), \sharp(\Gamma, B[\nuord \gamma x . B/x])) \mid \ordO{b} = \beta, \ \ordO{c} = \gamma, \ c \lessO b \le_{\KleeneO} a \}. 
\end{align*}
They both are recursively enumerable, hence definable by $\Sigma^0_1$-formulae. Define:
\[
\mathcal{D}^\prime_{\nu}(\UniNumA,\UniSetA) \df \exists b \ (b,\UniNumA) \in P \mathrel{\&} \forall c (c \lessO b) \to \exists y \big((b, c, \UniNumA, y) \in P \mathrel{\&} y \in \UniSetA\big)
\]
This formula is positive. It specifies that $\UniNumA = \sharp(\Gamma, \nuord \beta x . B)$ and, for each $\gamma<\beta$, $\sharp(\Gamma, B[\nuord \gamma x . B/x]) \in \UniSetA$. Introducing similar formulae for other rules and taking their disjunction yields a desired positive formula representing $\mathcal{D}^\prime(\UniNumA,\UniSetA)$. 

Define $\mathcal{D}(\UniNumA,\UniSetA) \df (\UniNumA \in \UniSetA) \vee \mathcal{D}^\prime(\UniNumA,\UniSetA)$.\anupam{Why do you need to add the left disjunct? Usually you do not need to do such `inflation'.} 
By induction on $\rk \Gamma$, it is easy to prove that $\Gamma$ is provable in $\muMALLordcf \alpha$ iff $\sharp\Gamma \in [\mathcal{D}]^{\rank(\Gamma)}$. This follows directly from Proposition \ref{proposition:ranks}. 
Therefore, if $b$ is a notation of $\beta = \omega^{\alpha^\omega}$, then $[\mathcal{D}]^\beta = \{\sharp \Gamma \mid \muMALLordcf \alpha \vdash \Gamma\}$. 
This is exactly the set of notations of theorems of $\muMALLord\alpha$, by \cref{cut-elim-mumall}; so let us denote it by $\mathrm{Th}_\alpha = \{\sharp \Gamma \mid \muMALLord \alpha \vdash \Gamma\}$.
Note that $\mathrm{Th}_\alpha = [\mathcal{D}^\prime](\mathrm{Th}_\alpha)$ because each theorem is obtained by some rule application from other theorems.\anupam{perhaps the syntax/semantics distinction is overly formal here. We can just speak of positive operators on $\mathbb N$ once defined, and we can simply identify $\muMALLord\alpha$ with $\mathrm{Th}_\alpha$.}

By \cref{theorem:reducibility-inductive-definitions}, $\mathrm{Th}_\alpha$ is computably enumerable w.r.t. $H(b)$.
What we need is to show that $\mathrm{Th}_\alpha$ is computable in $H(b)$, which is equivalent to $\mathrm{Th}_\alpha \le_\Tred H(b)$. It remains to show that the complement of $\mathrm{Th}_\alpha$ is computably enumerable in $H(b)$. Define the predicate $\overline{\mathcal{D}}(\UniNumA,\UniSetA) \df (\UniNumA \in \UniSetA) \vee \lnot \mathcal{D}^\prime(\UniNumA,\omega \setminus \UniSetA)$, which is also expressible by an positive formula.
\begin{proposition}\label{proposition:bar-D}
    We have the following:
    \begin{enumerate}
        \item\label{item:dual-approxs-disjoint-from-theory} For each $\gamma \in \Ord$, $[\overline{\mathcal{D}}]^\gamma \cap \mathrm{Th}_\alpha = \emptyset$. 
        \item\label{item:dual-op-contains-junk-codes} If $\UniNumA \ne \sharp \Gamma$ for any $\Gamma$, then $\UniNumA \in [\overline{\mathcal{D}}]$.
        \item\label{item:unprov-seq-enters-some-dual-approx} If $\muMALLord \alpha \not\vdash \Gamma$, then $\sharp\Gamma \in [\overline{\mathcal{D}}]^{\rank(\Gamma)+1}$.
        \item\label{item:compl-of-th-enters-dual-fp} $\omega \setminus \mathrm{Th}_\alpha = [\overline{\mathcal{D}}]^{\beta}$, when $\beta \geq \omega^{\alpha^\omega}$.
    \end{enumerate}
\end{proposition}
\begin{proof}
    \cref{item:dual-approxs-disjoint-from-theory} is proved by induction on $\gamma$. The case $\gamma = 0$ or $\gamma$ being limit is trivial. Let $\gamma = \delta+1$ and $\UniNumA \in [\overline{\mathcal{D}}]^\gamma = [\overline{\mathcal{D}}]([\overline{\mathcal{D}}]^\delta)$. This means that $\mathbb{N} \vDash (\UniNumA \in [\overline{\mathcal{D}}]^\delta) \vee \lnot \mathcal{D}^\prime(\UniNumA,\omega\setminus [\overline{\mathcal{D}}]^\delta)$. By the induction hypothesis, if $\UniNumA \in [\overline{\mathcal{D}}]^\delta$, then $\UniNumA \notin \mathrm{Th}_\alpha$. 
    Consider the case $\mathbb{N} \vDash \lnot \mathcal{D}^\prime(\UniNumA,\omega\setminus [\overline{\mathcal{D}}]^\delta)$. Assume, by way of contradiction, that $\UniNumA \in \mathrm{Th}_\alpha = [\mathcal{D}^\prime](\mathrm{Th}_\alpha)$. This implies that $\mathbb{N} \vDash \mathcal{D}^\prime(\UniNumA,\mathrm{Th}_\alpha)$. By the induction hypothesis, $\mathrm{Th}_\alpha \subseteq \omega \setminus [\overline{\mathcal{D}}]^\delta$, therefore, $\mathbb{N} \vDash \mathcal{D}^\prime(\UniNumA,\omega \setminus [\overline{\mathcal{D}}]^\delta)$. This is a contradiction.

    To prove \cref{item:dual-op-contains-junk-codes}, note that, for any $\UniSetA$, $\mathcal{D}^\prime(\UniNumA,\UniSetA)$ is true only if $\UniNumA$ is the notation of a sequent; therefore, if $\UniNumA$ is not a notation, then $\lnot \mathcal{D}(\UniNumA,\omega)$ holds, meaning that $\UniNumA \in [\overline{\mathcal{D}}]$.

    \cref{item:unprov-seq-enters-some-dual-approx} is proved by induction on $\rank(\Gamma)$. The case $\rank(\Gamma)=0$ is trivial since the rank of any sequent is at least $1$. Let $\rank(\Gamma)>0$. Assume, by way of contradiction, that $\sharp\Gamma \notin [\overline{\mathcal{D}}]^{\rank(\Gamma)+1}$, or equivalently, that: 
    \[
    \mathbb{N} \vDash (\UniNumA \notin [\overline{\mathcal{D}}]^{\rank(\Gamma)}) \mathrel{\&} \mathcal{D}(\sharp \Gamma, \omega \setminus [\overline{\mathcal{D}}]^{\rank(\Gamma)}).
    \]
    This implies that $\Gamma$ is a conclusion of a rule application of the form,
    \[
    \vlinf{}{}{\Gamma}{(\Delta_i)_{i<\zeta}}
    \]
    for some $\zeta$ such that $\sharp \Delta_i \notin [\overline{\mathcal{D}}]^{\rank(\Gamma)}$ for each $i < \zeta$. Thanks to \cref{proposition:ranks}, $\rank(\Delta_i) < \rank(\Gamma)$, hence $\sharp \Delta_i \notin [\overline{\mathcal{D}}]^{\rank(\Delta_i)+1}$. By the induction hypothesis, $\muMALLord \alpha \vdash \Delta_i$, and therefore $\Gamma$ is derivable in $\muMALLord \alpha$ as well using the above rule application. This is a contradiction.

    Finally \cref{item:compl-of-th-enters-dual-fp} follows from \cref{item:dual-op-contains-junk-codes,item:unprov-seq-enters-some-dual-approx} since ranks of all sequents are less than $\beta = \omega^{\alpha^\omega}$.
\end{proof}

 \cref{proposition:bar-D,theorem:reducibility-inductive-definitions} imply that $\omega \setminus \mathrm{Th}_\alpha$ (as well as $\mathrm{Th}_\alpha$) is computably enumerable in $H(b)$ when $b$ is a notation for $\omega^{\alpha^\omega}$. 
 Thus we have shown the membership direction of our main result, \cref{theorem:complexity-main}: 

 \begin{theorem}
 [Membership]
     \label{membership}
     Provability in $\muMALLord\alpha$ is in $\Sigma^0_{\omega^{\alpha^\omega}}$, for computable ordinals $\alpha$.
 \end{theorem}

\section{Lower bound on provability in $\muMALLord\alpha$}
\label{sec:lower-bd}

Let us now turn towards proving the lower bound for \cref{theorem:complexity-main}, which is much more involved.
The main result of this section is:

\begin{theorem}
[Hardness]
\label{theorem:lower-bound}
    Provability in $\muMALLord \alpha$ is $\Sigma^0_{\omega^{\alpha^\omega}}$-hard, for computable ordinals $\alpha \geq \omega$.
\end{theorem}

\noindent
To this end, for the remainder of this section, let us fix a computable ordinal $\alpha \ge \omega$, and thereby the system $\muMALLord \alpha$.

For the proof, we need a handy tool of hyperarithmetical theory: computable infinitary formulae which are built from the truth $\top$ and falsity $\bot$ constants by means of computable infinitary conjunctions and disjunctions. 
We shall follow mostly the exposition of \cite{AshK2000}.

\begin{definition}
[Computable infinitary propositional sentences]
\label{definition:ips}
    For $a \in \KleeneO$, define: 
    \[
    S^{\Sigma}_{a}\ :=\
    {\left\{\langle 0, a, e \rangle \mid e \in \omega \right\}}
    \quad \text{and} \quad
    S^{\Pi}_{a}\ :=\
    {\left\{\langle 1, a, e \rangle \mid e \in \omega \right\}} .
    \]
    Elements of these sets are called \textbf{indices}. Indices are mapped to \emph{infinitary propositional sentences} by means of a function $\ips{\cdot}$ defined by:
    \begin{itemize}
        \item 
        For $\varepsilon \in \{0,1\}$, let $\ips{\langle \varepsilon, 1, e \rangle} = 
        \begin{cases}
            \bot & e=0; \\
            \top & e>0. \\
        \end{cases}
        $
        \item For $a \ne 1$, let $\ips{\langle 0, a, e \rangle}\df $ 
        \[
        \bigvee
        \left\{ \ips{k} \mid {\UCF\left(e,k\right)}\ \text{converges and}\ {\left( \exists b \lessO a \right)}\, {\left( \exists e^\prime \right)}\, {\left( k = \langle 1, b, e^\prime \rangle \right)} \right\} .
        \]
        \item Symmetrically, for $a \ne 1$, let $\ips{\langle 1, a, e \rangle} \df$
        \[
        \bigwedge
        \left\{ \ips{k} \mid {\UCF\left(e,k\right)}\ \text{converges and}\ {\left( \exists b \lessO a \right)}\, {\left( \exists e^\prime \right)}\, {\left( k = \langle 0, b, e^\prime \rangle \right)} \right\} .
        \]
    \end{itemize}
    The definition of truth of an infinitary propositional sentence is classical: $\top$ is true, $\bot $ is not true, a disjunction is true iff one of the disjuncts is true, a conjunction is true iff all the conjuncts is true.
    Let $\TrueIPS{a}\df \{\langle \varepsilon, b, e \rangle \mid b \lessO a, \ \ips{\langle \varepsilon, b, e \rangle} ~ \text{is true} \}$ be the set of indices of true infinitary propositional sentences with ordinal notations below $a$.
\end{definition}

The main fact relating computable infinitary sentences to hyperarithmetical hierarchy is presented below. 

\begin{fact}
    Let $b \in \KleeneO$ represent an infinite limit ordinal $\ordO{b}$. Then $\TrueIPS{b}$ is many-one equivalent to $H(b)$.
\end{fact}

See e.g.~\cite{AshK2000} for a proof. In particular, the fact that $H(b)$ is many-one reducible to $\TrueIPS{b}$, proved in \cite[Proposition 7.8(a)]{AshK2000}, shows that $\TrueIPS{b}$ is $\Sigma^0_{\ordO{b}}$-hard. 

Our main goal in this section is to provide a Turing reduction from $\TrueIPS{c}$ to the $\muMALLord \alpha$ provability predicate for some $c$ a notation for $\omega^{\alpha^\omega}$.

\subsection{Encoding computable functions in $\muMALLord \alpha$}

First, in order to reduce $\TrueIPS{c}$ to provability in $\muMALLord\alpha$, we want to simulate computable functions within $\muMALLord \alpha$. This is done in the spirit of the constructions from \cite{DasDS23,Kuznetsov22}: we take a Minsky machine that implements a computable function and encode its instructions in $\muMALLord \alpha$.

\begin{definition}
    A \textbf{Minsky machine} is a tuple $\MinMach = (Q,s,t,I)$ where $Q$ is a set of states, $s,t \in Q$ are the start state and the accepting state resp., and $I$ is a finite set of instructions. An instruction is either of the form $\INC(p,i,q)$ or of the form $\JZDEC(p,i,q,r)$ where $p,q,r \in Q$ and $i \in \omega$. 
    
    A Minsky machine's \textbf{configuration} is of the form $p(k_0,k_1,\ldots,k_{n-1})$ where $p \in Q$ and $c_i \in \omega$.
    Operational semantics, mapping a configuration to the next one, is given by:
    \begin{itemize}
        \item The semantics of the $\INC(p,i,q)$ is: from the state $p$, increase $k_i$, i.e.~the value of the counter No.~$i$, by $1$ and switch to the state $q$. 

    \item
    The semantics of $\JZDEC(p,i,q,r)$ is: from the state $p$, if $k_i$ is greater than $0$, decrease it by $1$ and switch to the state $q$; otherwise, do not change counter values and switch to the state $r$.
    \end{itemize}

    The Minsky machine $\MinMach$ \textbf{computes} a partial function $f : \underbrace{\omega \times \ldots \times \omega}_{k~\text{times}} \parto \underbrace{\omega \times \ldots \times \omega}_{l~\text{times}}$ 
    if the following two statements are equivalent:
    \begin{enumerate}
        \item $t(m_0,m_1,\ldots)$ is reachable from $s(n_0,\ldots,n_{k-1},0,0,\ldots)$;
        \item $f(n_0,\ldots,n_{k-1})$ is defined and equals $(m_0,\ldots,m_{l-1})$, while $m_i = 0$ for $i \ge l$.
    \end{enumerate}
\end{definition}

\begin{fact}
[\cite{Minsky67}]
    Any computable function is computed by some Minsky machine.
\end{fact}

Now let us set up some notation for formulas in $\muMALLord \alpha$. 

\begin{itemize}
    \item Fix three propositional symbols, denoted by $\VARacc$, $\VARkey_1$ and $\VARkey_2$, and also countably many symbols, denoted by $(r_i)_{i \in \omega}$.
    \item Let $\tup(n_0,\ldots,n_k)=\VARsym_0^{n_0},\ldots,\VARsym_k^{n_k}$.
    (The symbol $\VARsym_i$ repeated $k$ times corresponds to the $i$\textsuperscript{th} counter of a Minsky machine being equal to $k$.)
    \item Let $\FORPass \df \neg \VARacc \vlte \VARacc$.
    \item Let $\FORAcc \df \neg \VARacc \vlpa \VARacc$ (note that $\FORAcc$ is the negation of $\FORPass$).
\end{itemize}

\begin{definition}
    Let $\Xi \subseteq \Prop$. A formula is \textbf{$\Xi$-locked} if it has the form, 
    \[
        (\neg p_1 \vlte A_1) \vlor \ldots \vlor (\neg p_m \vlte A_m) \vlor \FORPass
    \]
    where each $p_i \in \Xi$.
    A sequent $\Gamma$ is \textbf{$\Xi$-locked} if all its formulae are $\Xi$-locked.
\end{definition}
Informally, elements of $\Xi$ are `keys', and $p_i$ is required to `unlock' the $i$\textsuperscript{th} disjunct of the above formula (while $\VARacc$ `unlocks' the last disjunct). 

\begin{lemma}\label{lemma:FORComp}
    For any $l$ partial computable functions $f_1,\ldots, f_l$, there exists a formula $\FORComp$ and $s_1,\ldots,s_l,t \in \Prop$ such that, for any $\{t,\VARkey_1\}$-locked multiset $\Gamma$ and any $d \in \{1,\ldots, l\}$, the sequent, 
    \[
    \Gamma, \tup(\vec{n}),s_d, \FORComp
    \]
    is derivable if and only if $f_d(\vec{n})$ is defined and the sequent,
    \[
    \Gamma, \tup(f_d(\vec{n})),t
    \]
    is derivable.
\end{lemma}
\begin{proof}
    For each $j=1,\ldots,l$, fix a Minsky machine $\MinMach_j = (Q_j,s_j,t_j,I_j)$ that computes $f_j$. Without loss of generality, we assume that the sets $Q_1,\ldots,Q_l$ are disjoint. There is some $N \in \omega$ such that the counters No.~$(N+1)$, No.~$(N+2)$ etc.~are used by neither of $\MinMach_i$'s (i.e.~the instructions of $\MinMach_i$'s involve only the first $(N+1)$ counters).
    
    Fix a variable $\UniNumA$. We encode a Minsky machine instruction $\mathit{ins}$ by a formula $\encIns{\ins}$ defined by,
    \begin{itemize}
        \item $\encIns{\INC(p,i,q)} \df \neg p \vlte (\VARsym_i \vlpa q \vlpa x)$
        \item $\encIns{\JZDEC(p,i,q,r)} \df \neg p \vlte [(\neg \VARsym_i  \vlte (q \vlpa x)) \vlor (Z_i \vlan (r \vlpa x))]$
    \end{itemize}
    where $Z_i$ is the ``zero-check'' formula given by:
    \[
    Z_i \df \mu y \left(\FORAcc \vlor (\neg r_0 \vlor \ldots \vlor \neg r_{i-1} \vlor \neg r_{i+1} \vlor \ldots \vlor \neg r_N) \vlte y\right).
    \]
    Now define,
    \[
    \FORComp_\beta \df \muord \beta x \left[T \vlor \bigoplus\limits_{j=1}^l\bigoplus\limits_{\ins \in I_j} \encIns{\ins} \right]
    \]
    where $T \df (\neg t_1 \vlte t) \vlor \ldots \vlor (\neg t_l \vlte t)$ and $t$ is a fresh symbol. The desired formula $\FORComp$ is defined as $\FORComp_\alpha$.

    The desired statement follows from a more general lemma.
    \begin{lemma}\label{lemma:FORComp-beta}
        Given $p \in Q_d$, consider a sequent of the form 
        \begin{equation}\label{sequent:FORComp-beta}
            \Gamma, \tup(\vec{n}),p, \FORComp_\beta.
        \end{equation}
        It is provable if and only if so is a sequent of the form 
        \[
        \Gamma, \tup(\vec{m}),t
        \]
        such that $t_d(\vec{m},0,0,\ldots)$ is reachable from $p(\vec{n},0,0,\ldots)$---in less than $\beta$ steps if $\beta$ is finite (while the number of steps can be arbitrary for infinite $\beta$).
    \end{lemma}
    The proof of the latter, placed in \cref{proof:lemma:FORComp-beta}, is by induction on $\beta$, and it consists of a straightforward analysis of focussed proofs, in the system $\muMALLFord\alpha$, of \eqref{sequent:FORComp-beta}, exploiting \cref{focussing-completeness}. Finally, take $p=s_d$ and $\beta=\alpha$ in the statement of \cref{lemma:FORComp-beta}. Since $\MinMach_d$ computes $f_d$, the statement of \cref{lemma:FORComp} follows.
\end{proof}

It is pertinent to point out that we now already obtain:

\begin{proposition}\label{proposition:Sigma01-hardness}
    Provability in $\muMALLord \alpha$ is $\Sigma^0_1$-hard. 
\end{proposition}
\begin{proof}
    Take a $\Sigma^0_1$-complete set $A$ and consider the function
    \[
    f_1(n) \df 
    \begin{cases}
        0 & n \in A ; \\
        \text{undefined} & n \notin A . \\
    \end{cases}
    \]
    Applying \cref{lemma:FORComp} from here, we have that the sequent $\FORPass \vlor (\neg t \vlte \mathbf{1}), \VARsym_0^n,s_1,\FORComp$ is provable iff $n \in A$.
\end{proof}

\subsection{Encoding computable infinitary sentences in $\muMALLord \alpha$}

The second stage of our construction is simulating the inductive definition of truth of an infinitary proposition sentence within $\muMALLord \alpha$. The following lemma serves this purpose.

\begin{lemma}\label{lemma:FORInd}
    Let $a_0 \in \KleeneO$. There exists a formula $\FORInd$ such that, for any $\{\VARkey_1\}$-locked multiset $\Gamma$ and any $\UniNumA \in \omega$, the sequent
    \[
    \Gamma, \VARsym_0^\UniNumA, \FORInd
    \]
    is derivable iff $\UniNumA = \langle \varepsilon, a, e\rangle$ and one of the following holds:
    \begin{enumerate}
        \item $\ips{\UniNumA} = \top$ (i.e.~$\ordO{a}=0$ and $e>0$);
        \item $1 \lessO a \lessO a_0$ and there is an index $\UniNumB \in \omega$ of a disjunct of $\ips{\UniNumA}$ such that the sequent
        \[
        \Gamma, \VARsym_0^\UniNumB, \VARkey_1
        \]
        is derivable;
        \item $1 \lessO a \lessO a_0$ and, for each $\UniNumB \in \omega$ being an index of a conjunct of $\ips{\UniNumA}$, the sequent 
        \[
        \Gamma, \VARsym_0^\UniNumB, \VARkey_1 
        \]
        is derivable.
    \end{enumerate}
\end{lemma}

\begin{proof}

First, let us apply \cref{lemma:FORComp} to the following three computable functions:
\begin{enumerate}
    \item Let $f_1(\UniNumA) \df 0$ iff $\ips{\UniNumA} = \top$ (undefined otherwise).
    \item Let $f_2(\UniNumA) \df (\UniNumA,1,0)$ iff $\UniNumA \in S^\Sigma_a$ for some $a$ such that $1 \lessO a \lessO a_0$; let $f_2(\UniNumA) \df (\UniNumA,0,1)$ iff $\UniNumA \in S^\Pi_a$ for some $a$ such that $1 \lessO a \lessO a_0$ (undefined otherwise).
    \item Let us fix a decidable ternary predicate $\Witn \subseteq \omega \times \omega \times \omega$ such that $\exists \UniNumC \, \Witn(\UniNumA,\UniNumB,\UniNumC)$ is true iff $\UniNumB$ is an index of a conjunct/disjunct of $\ips{\UniNumA}$. Let $f_3(\UniNumA,\UniNumB,\UniNumC) \df (\UniNumB,1,0)$ iff $\Witn(\UniNumA,\UniNumB,\UniNumC)$ is true, otherwise $f_3(\UniNumA,\UniNumB,\UniNumC) \df (0,0,1)$.
\end{enumerate}
Note that $\Witn(\cdot,\cdot,\cdot)$ exists because the predicate ``$\UniNumB$ is an index of a conjunct/disjunct of $\ips{\UniNumA}$'' is recursively enumerable.

Apply \cref{lemma:FORComp} to obtain a formula $\FORComp$ and propositional symbols $s_1,s_2,s_3,t$ satisfying the corresponding properties. 
Now, we use these formulae to define the required formula $\FORInd$:
\begin{align*}
    \FORInd &\df \FORBase \vlor \FORStep \\
    \FORBase &\df s_1 \vlpa \FORComp \vlpa (\neg t \vlte \FORAcc) \\
    \FORStep &\df s_2 \vlpa \FORComp \vlpa (\neg t \vlte ((\neg \VARsym_1 \vlte \FORStep_\exists)\vlor (\neg \VARsym_2 \vlte \FORStep_\forall))) \\
    \FORStep_\exists &\df \mu x. ((\VARsym_1 \vlpa x) \vlor (\VARsym_2 \vlpa x) \vlor \VARkey_2) \vlpa \FORNext_\exists \\
    \FORStep_\forall &\df \nu x. ((\VARsym_1 \vlpa x) \vlan (\VARsym_2 \vlpa x) \vlan \VARkey_2) \vlpa \FORNext_\forall \\
    \FORNext_\exists &\df \neg \VARkey_2 \vlte (s_3 \vlpa \FORComp \vlpa (\neg t \vlte \neg \VARsym_1 \vlte \VARkey_1)) \\
    \FORNext_\forall &\df \neg \VARkey_2 \vlte (s_3 \vlpa \FORComp \vlpa (\neg t \vlte ((\neg \VARsym_1 \vlte \VARkey_1) \vlor (\neg \VARsym_2 \vlte \FORAcc))))\
\end{align*}
Again, the construction's correctness is proved by a straightforward (and tedious) analysis of focussed proofs, in the system $\muMALLFord\alpha$, exploiting \cref{focussing-completeness}. 
We provide a proof of correctness in \cref{proof:lemma:FORInd}.
\end{proof}

Let us explain the role of each of the formulae occurring in the proof of \cref{lemma:FORInd} at an informal level:
\begin{itemize}
    \item The formula $\FORBase$ checks whether $\ips{\UniNumA}$ is $\top$ and outputs the accepting formula $\FORAcc$ if this is the case. 
    \item
    The formula $\FORStep$ checks whether $\UniNumA \in S^\Sigma_a \cup S^\Pi_a$ for $1 
    \lessO \ordO{a} \lessO a_0$ or not and then switches to $\FORStep_\exists$ ($\FORStep_\forall$) if $\UniNumA\in S^\Sigma_a$ ($\UniNumA\in S^\Pi_a$ respectively). 
    \item
    The formula $\FORStep_\exists$ generates $\VARsym_1^\UniNumB$ and $\VARsym_2^\UniNumC$ for some $\UniNumB,\UniNumC \in \omega$. Then, $\FORNext_\exists$ checks whether $\Witn(\UniNumA,\UniNumB,\UniNumC)$ is true and, if it is so, outputs $\VARsym_0^\UniNumB$. This corresponds to choosing the disjunct $\ips{\UniNumB}$ of the formula $\ips{\UniNumA}$ and forming a sequent $\Gamma, \VARsym_0^\UniNumB, \VARkey_1$ for it. 
    \item 
    Dually, the formula $\FORStep_\forall$ generates $\VARsym_1^\UniNumB,\VARsym_2^\UniNumC$ for all $\UniNumB,\UniNumC \in \omega$. Then, $\FORNext_\forall$ checks whether $\Witn(\UniNumA,\UniNumB,\UniNumC)$ is true. If it is false, $\FORNext_\forall$ outputs $\FORAcc$, and it outputs $\VARsym_0^\UniNumB$ otherwise. This means that we form the sequent $\Gamma, \VARsym_0^\UniNumB, \VARkey_1$ for each conjunct $\ips{\UniNumB}$ of the formula $\ips{\UniNumA}$ while we immediately accept if $\ips{\UniNumB}$ is not a conjunct of $\ips{\UniNumA}$.
\end{itemize}

\subsection{Proof of hardness}

\tikhon{focussing: the rightmost formula of the right zone must be active.}

Let us now start assembling the proof of our hardness result, \cref{theorem:lower-bound}.
We start with a folklore fact about Kleene's $\KleeneO$:

\begin{fact}
    There exist computable partial functions $\mult(\cdot,\cdot)$ and $\exp(\cdot,\cdot)$ such that, if $a,b \in \KleeneO$, then $\mult(a,b), \exp(a,b) \in \KleeneO$ and $\ordO{\mult(a,b)} = \ordO{a}\cdot{\ordO{b}}$, $\ordO{\exp(a,b)} = \ordO{a}^{\ordO{b}}$. 
\end{fact}
This fact is left as Exercise 11-58 in \cite{Rogers67} and is proved in the same way as existence of addition on ordinal notations \cite[Theorem XVII]{Rogers67}.

Recall that, at the beginning of the section, we fixed some computable $\alpha \ge \omega$ (and system $\muMALLord\alpha$). 
Let us now fix $a_1,a_2 \in \KleeneO$ such that $\ordO{a_1}=\omega$ and $\ordO{a_2}=\deg\alpha$. Let $e_0$ be the index of the computable function $g$ defined as follows:
\begin{align}\label{equation:g(n)}
    g(n) &\df \exp(a_1,\exp(a_1,\mult(a_2,2\uparrow\uparrow n)) \ \text{for $n \in \omega$}
\end{align}
Recall that $2 \uparrow\uparrow n = \underbrace{2^{2^{\ldots^2}}}_{n~\text{times}}$ is the unique notation of the natural number $n$. Clearly, $\ordO{g(n)} = \omega^{\omega^{(\deg\alpha) n}}$. 

Let $a_0 \df 3 \cdot 5^{e_0}$. This is a notation for $\omega^{\alpha^\omega}$, by \cref{proposition:omega-power}. Plugging this $a_0$ into \cref{lemma:FORInd}, we obtain the formula $\FORInd$ satisfying the properties stated in that lemma. 
Using it, let us prepare ingredients for the main construction. 

Recall from \cref{definition:rho-formula} the formulas $\Resource_n$, for $n<\omega$, given by $\Resource_0 = x_0$ and $\Resource_{n+1} = \muord \alpha x_{n}.\left(\UniNumA_{n+1} \vlpa \Resource_n\right)$.

\begin{itemize}
    \item Let $\lock(A) \df \FORPass \vlor(\neg \VARkey_1 \vlte A)$ for each $A \in \Fm$;
    \item Let us define a function $\aug(\cdot)$ whose domain is the set of formulae built using $\vlpa$ and $\muord \beta$ only:
    \begin{itemize}
        \item $\aug(c) \df \lock(c)$ for $c$ being a variable, a symbol or a constant;
        \item $\aug(B_1 \vlpa B_2) \df \lock(\FORInd \vlpa \aug(B_1) \vlpa \aug(B_2))$;
        \item $\aug(\muord \beta y. D) \df \lock \left( \muord \beta y. (\FORInd \vlpa  \aug(D))\right)$.
        \end{itemize}
    \item Let $\hat{D}$ be obtained from $\aug(D)$ by substituting $\FORInd$ for $x$ for each $x \in \Var$.
    \item Let $\rightsquigarrow$ be the binary relation defined as follows: $\muord \beta x\, A(x) \rightsquigarrow A(\muord \gamma x\, A(x))$ for $\gamma < \beta$; $A_1 \vlpa A_2 \rightsquigarrow A_i$ for $i=1,2$. Let $\mathcal{E} \df \{A \mid \exists n\ \Resource_n \rightsquigarrow^\ast A\}$.
\end{itemize}


\begin{lemma}\label{lemma:lower-bound-main}
    Given $E_1,\ldots,E_l \in \mathcal{E}$, consider the sequent:
    \begin{equation}\label{sequent:main}
        \hat E_1,\ldots,\hat E_l, \VARsym_0^{\UniNumA}, \VARkey_1
    \end{equation}
    \begin{enumerate}
        \item\label{item:deriv-imp-true} If (\ref{sequent:main}) is derivable, then $\UniNumA \in \TrueIPS{a_0}$;
        \item\label{item:true-imp-deriv} If $\UniNumA = \langle \varepsilon,a,e\rangle \in \TrueIPS{a_0}$ and $\rank(E_1,\ldots,E_l) > \ordO{a}$, then (\ref{sequent:main}) is derivable.
    \end{enumerate}
\end{lemma}

\begin{proof}
The proof is by induction on $\zeta = \rank(E_1,\ldots,E_l)$. 
    If $\zeta=0$, the rank of each $E_i$ is at least $1$, so $l=0$. In this case, (\ref{sequent:main}) is underivable, and the antecedents of both statements of the lemma are false. 

    So suppose $\zeta>0$ and let us prove \cref{item:deriv-imp-true}. 
    Assume that the sequent (\ref{sequent:main}) is provable; equivalently, b, there is a focussed proof of the sequent
    \[
    \hat E_1,\ldots,\hat E_l, \VARsym_0^{\UniNumA}, \VARkey_1  \focusedNeg \emptyZone
    \]
    in $\muMALLFord \alpha$. 
    Any proof of this sequent ends with a decision step, $\decide$. Without loss of generality, we can assume that $\hat E_1$ goes to the right zone of its premisse.

    \begin{itemize}
        \item \textit{Case 1.} $E_1$ is a variable. Then, $\hat E_1 = \lock(\FORInd)$ and the $\muMALLFord \alpha$ derivation ends as follows:
    \[
        \vlderivation{
        \vlin{\decide}{}
        {
            \hat E_1,\ldots,\hat E_l, \VARsym_0^{\UniNumA}, \VARkey_1  \focusedNeg \emptyZone
        }
        {
            \vlin{\vlor}{}
            {
                \hat E_2,\ldots,\hat E_l, \VARsym_0^{\UniNumA}, \VARkey_1 \focusedPos \hat E_1
            }
            {
                \vliin{\vlte}{}
                {
                    \hat E_2,\ldots,\hat E_l, \VARsym_0^{\UniNumA}, \VARkey_1 \focusedPos \VARkey_1^\bot \otimes \FORInd
                }
                {
                    \vlhy{\VARkey_1 \focusedPos \VARkey_1^\bot}
                }
                {
                    \vlhy{\hat E_2,\ldots,\hat E_l, \VARsym_0^{\UniNumA} \focusedPos \FORInd}
                }
            }
        }}
    \]
    Let us denote $E_2, \ldots, E_l$ by $\Theta$. The above derivation implies that $\hat \Theta, \VARsym_0^{\UniNumA} , \FORInd$ is provable in $\muMALLord\alpha$. Note that $\rank(\Theta) < \zeta$.
    \item 
    \textit{Case 2.} $E_1$ is of the form $E^\prime \vlpa E^{\prime\prime}$. Then $\hat E_1 = \lock(\FORInd \vlpa \hat E^\prime \vlpa \hat E^{\prime\prime})$ and the $\muMALLFord \alpha$ derivation ends as follows:
    \[
        \vlderivation{
        \vlin{\decide}{}
        {
            \hat E_1,\ldots,\hat E_l, \VARsym_0^{\UniNumA}, \VARkey_1  \focusedNeg \emptyZone
        }
        {
            \vlin{\vlor}{}
            {
                \hat E_2,\ldots,\hat E_l, \VARsym_0^{\UniNumA}, \VARkey_1 \focusedPos \hat E_1
            }
            {
                \vliin{\vlte}{}
                {
                    \hat E_2,\ldots,\hat E_l, \VARsym_0^{\UniNumA}, \VARkey_1 \focusedPos \VARkey_1^\bot \otimes (\FORInd \vlpa \hat E^\prime \vlpa \hat E^{\prime\prime})
                }
                {
                    \vlhy{\VARkey_1 \focusedPos \VARkey_1^\bot}
                }
                {
                    \vliq{}{}
                    {
                        \hat E_2,\ldots,\hat E_l, \VARsym_0^{\UniNumA} \focusedPos \FORInd \vlpa \hat E^\prime \vlpa \hat E^{\prime\prime}
                    }
                    {
                        \vlhy{\hat E^\prime , \hat E^{\prime\prime}, \hat E_2,\ldots,\hat E_l, \VARsym_0^{\UniNumA}, \FORInd \focusedNeg \emptyZone}
                    }
                }
            }
        }}
    \]
    Let $\Theta \df E^\prime , E^{\prime\prime}, E_2,\ldots, E_l$ in this case. Again, $\rank(\Theta)<\zeta$ because $\rank(E^\prime) \natsum \rank(E^{\prime\prime}) < \rank(E_1)$.
\item
    \textit{Case 3.} $E_1 = \muord \beta x.F(x)$. Then, $\hat E_1 = \lock \left( \muord \beta x. (\FORInd \vlpa  \aug(F(x)))\right)$. Let us denote $\FORInd \vlpa  \aug(F(x))$ by $D(x)$.  The derivation ends as follows for some $\gamma<\beta$:
    \[
        \vlderivation{
        \vlin{\decide}{}
        {
            \hat E_1,\ldots,\hat E_l, \VARsym_0^{\UniNumA}, \VARkey_1  \focusedNeg \emptyZone
        }
        {
            \vlin{\vlor}{}
            {
                \hat E_2,\ldots,\hat E_l, \VARsym_0^{\UniNumA}, \VARkey_1 \focusedPos \hat E_1
            }
            {
                \vliin{\vlte}{}
                {
                    \hat E_2,\ldots,\hat E_l, \VARsym_0^{\UniNumA}, \VARkey_1 \focusedPos \VARkey_1^\bot \otimes \muord \beta x. D(x)
                }
                {
                    \vlhy{\VARkey_1 \focusedPos \VARkey_1^\bot}
                }
                {
                    \vlin{\muord \beta}{}
                    {
                        \hat E_2,\ldots,\hat E_l, \VARsym_0^{\UniNumA} \focusedPos \muord \beta x. D(x)
                    }
                    {
                        \vliq{}{}
                        {
                            \hat E_2,\ldots,\hat E_l, \VARsym_0^{\UniNumA} \focusedPos D(\muord \gamma x. D(x))
                        }
                        {
                            \vlhy{\aug(F(x))[\muord \gamma x. D(x)/x], \hat E_2,\ldots,\hat E_l, \VARsym_0^{\UniNumA}, \FORInd \focusedNeg \emptyZone}
                        }
                    }
                }
            }
        }}
    \]
    Notice that the formula $\aug(F(x))[\muord \gamma x. D(x)/x]$ coincides with $\hat{G}$ for $G = F(\muord \gamma x. F(x))$. Let $\Theta \df G, E_2,\ldots, E_l$ in this case. Once again, $\rank(\Theta)<\zeta$ because $\rank(G) < \rank(E_1)$.

    \end{itemize}

    In each of the three cases, we have that the sequent $\hat \Theta, \VARsym_0^\UniNumA,\FORInd$ is derivable in $\muMALLord \alpha$. By \cref{lemma:FORInd}, its derivability implies that $\UniNumA = \langle \varepsilon, a, e\rangle$ and that one of the following holds:
    \begin{enumerate}[a.]
        \item $\ips{\UniNumA} = \top$ (consequently, $\UniNumA \in \TrueIPS{a_0}$);
        \item $1 \lessO a \lessO a_0$ and there is an index $\UniNumB \in \omega$ of a disjunct of $\ips{\UniNumA}$ such that the sequent
        $
        \hat \Theta, \VARsym_0^\UniNumB, \VARkey_1
        $
        is derivable;
        \item $1 \lessO a \lessO a_0$ and, for each $\UniNumB \in \omega$ being the index of a conjunct of $\ips{\UniNumA}$, the sequent 
        $
        \hat \Theta, \VARsym_0^\UniNumB, \VARkey_1 
        $
        is derivable.
    \end{enumerate}
    In the two latter situations, apply the induction hypothesis and conclude that $\ips{\UniNumB}$ is true; therefore, $\ips{\UniNumA}$ is true and hence $\UniNumA \in \TrueIPS{a_0}$.

    \medskip

    It remains to prove \cref{item:true-imp-deriv}. We are given that $\UniNumA = \langle \varepsilon,a,e \rangle$ where $a \lessO a_0$ such that $\ips{\UniNumA}$ is true. We consider two cases:
    \begin{itemize}
        \item \textit{Case A.} There is $i \in \{1,\ldots,l\}$ such that $\rank(E_i)$ is a successor ordinal. Without loss of generality, $l=1$. Depending on whether $E_1 \in \Var$, $E_1 = E^\prime \vlpa E^{\prime\prime}$ or $E_1 = \muord \beta x. F(x)$, consider the $\muMALLord \alpha$ derivation from Case 1, Case 2, and Case 3 above respectively and remove the focusing decorations $\focusedPos,\focusedNeg$ from it. In the derivation from Case 3, we additionally need to choose $\gamma$. It is given to us that $E_1 = \muord \beta x. F(x)$ and $\rank(E_1) = \sup\limits_{\delta< \beta} \left( \rank(F(\muord \delta x. F(x))) + 1 \right)$. Since the supremum of a family of ordinals cannot simultaneously be a successor ordinal and be greater than all ordinals from the family, we choose an ordinal $\gamma < \beta$ such that $\rank(E_1) = \rank(F(\muord \gamma x. F(x))+1$. 

    In each of the three cases, the goal sequent is derived from $\hat \Theta, \VARsym_0^\UniNumA,\FORInd$, and $\rank(\Theta)+1 = \rank(E_1,\ldots,E_l)$.
    Since $\ips{\UniNumA}$ is true, one of the following holds.
    \begin{enumerate}[i.]
        \item $\ips{\UniNumA} = \top$. By \cref{lemma:FORInd}, the sequent $\hat \Theta, \VARsym_0^\UniNumA,\FORInd$ is provable in $\muMALLFord \alpha$, hence so is the sequent (\ref{sequent:main}).
        \item $1 \lessO a \lessO a_0$, $\varepsilon=0$, and there is $\UniNumB = \langle 1, b, e^\prime \rangle \in \omega$ such that $\ips{\UniNumB}$ is a disjunct of $\ips{\UniNumA}$ and $\UniNumB \in \TrueIPS{a_0}$. Fix such $\UniNumB$. Note that $\ordO{b} < \ordO{a} < \rank(E_1,\ldots,E_l)$, therefore, $\ordO{b} <  \rank(\Theta)$. This enables one to apply the induction hypothesis and conclude that $\hat \Theta, \VARsym_0^\UniNumB,\VARkey_1$ is provable in $\muMALLFord \alpha$. By \cref{lemma:FORInd}, the sequent $\hat \Theta, \VARsym_0^\UniNumA,\FORInd$ is provable, hence so is (\ref{sequent:main}).
        \item $1 \lessO a \lessO a_0$, $\varepsilon=1$, and, for each $\UniNumB = \langle 1, b, e^\prime \rangle \in \omega$ such that $\ips{\UniNumB}$ is a conjunct of $\ips{\UniNumA}$, $\ips{\UniNumB}$ is true. Similarly, by the induction hypothesis, $\hat \Theta, \VARsym_0^\UniNumB,\VARkey_1$ is provable in $\muMALLFord \alpha$ for all such $\UniNumB$. By \cref{lemma:FORInd}, the sequent $\hat \Theta, \VARsym_0^\UniNumA,\FORInd$ is derivable, hence so is (\ref{sequent:main}).
    \end{enumerate}
\item
    \textit{Case B.} For each $i=1,\ldots,l$, $\rank(E_i)$ is limit. To proceed further, we shall use a form of continuity of natural sum. 
    \begin{definition}
        If $\delta = \omega^{\alpha_1} \cdot \beta_1+\ldots+\omega^{\alpha_k} \cdot \beta_k$ is the Cantor normal form of an ordinal $\delta$, i.e. $\beta_i < \omega$ and $\alpha_1>\ldots>\alpha_k$, then let $\ld(\delta) \df \alpha_k$ denote the \emph{lowest degree} in the expansion of $\delta$.
    \end{definition}
    \begin{fact}[{\cite[Remark 2.2]{Lipparini25}}]\label{fact:continuity-natural-sum}
        If $0<\ld(\delta_1) \le \ld(\delta_2)$, then $\delta_1 \natsum \delta_2 = \sup\limits_{\delta < \delta_1} (\delta \natsum \delta_2)$. 
    \end{fact}
    Without loss of generality, assume that
    \[
    0
    <
    \ld(\rank(E_1))
    \le
    \ldots
    \le
    \ld(\rank(E_l)).
    \]
    It must be the case that $E_1 = \muord \beta x. F(x)$, otherwise $\rank(E_1)$ is a successor ordinal or $0$. We know that $\rank(E_1,\ldots,E_l) > \ordO{a}$. \cref{fact:continuity-natural-sum} implies that there is $\gamma<\beta$ such that
    \begin{equation}\label{equation:gamma-choice}
        (\rank(F(\muord \gamma x. F(x)))+1) \natsum \rank(E_2) \natsum \ldots \natsum \rank(E_l) > \ordO{a}
    \end{equation}
    Consequently, 
    \[\rank(F(\muord \gamma x. F(x)),E_2,\ldots,E_l) \ge \ordO{a}
    \]
    Consider the derivation from Case 3 above with $\gamma$ satisfying \cref{equation:gamma-choice}. In this derivation (after removing the focusing decorations) the target sequent $\hat E_1,\ldots,\hat E_l, \VARsym_0^{\UniNumA}, \VARkey_1$ is derived from $\hat \Theta, \VARsym_0^\UniNumA,\FORInd$ where $\Theta = F(\muord \gamma x. F(x)),E_2,\ldots,E_l$.
    It remains to consider three cases, similar to i, ii, and iii above. \qedhere
    \end{itemize}
\end{proof}

At last, we are ready to prove \cref{theorem:lower-bound}.
\begin{proof}
    Our goal is to show that $\TrueIPS{a_0}$ is Turing-reducible to the set of sequents provable in $\muMALLord \alpha$. Given $\UniNumA < \omega$ as an input, first test whether $\UniNumA = \langle \varepsilon, a, e \rangle$ for $\varepsilon \in \{0,1\}$, $a \lessO a_0$, and $e \in \omega$. Since the set of such $\UniNumA$'s is computably enumerable and provability in $\muMALLord \alpha$ is $\Sigma^0_1$-hard (\cref{proposition:Sigma01-hardness}), we can perform such a test using provability in $\muMALLord \alpha$ as an oracle. If $\UniNumA$ is not of the above form, reject it. Otherwise, compute the least $n$ such that $a \lessO g(n)$, with $g(n)$ defined by \cref{equation:g(n)}. The inequality $a \lessO g(n)$ implies that $\ordO{a} < \ordO{g(n)} = \omega^{\omega^{(\deg \alpha)n}}$. Then, test provability of the sequent
    \begin{equation}\label{sequent:proof-theorem:lower-bound}
        \hat \Resource_{n+1}, \VARsym_0^\UniNumA, \VARkey_1
    \end{equation}
    and accept $\UniNumA$ if it is provable (otherwise, reject it).
    
    Recall that $\rank(\Resource_{n+1}) \ge \omega^{\omega^{(\deg \alpha)n}} > \ordO{a}$, by \cref{lemma:rho-ranks-lower-bound}. By \cref{lemma:lower-bound-main}, the sequent (\ref{sequent:proof-theorem:lower-bound}) is provable in $\muMALLFord \alpha$ iff $\UniNumA \in \TrueIPS{a_0}$.
\end{proof}

\section{Remaining Proofs}

\subsection{Proof of \cref{lemma:FORComp-beta}}\label{proof:lemma:FORComp-beta}

Induction on $\beta$. Since (\ref{sequent:FORComp-beta}) does not have negative formulae, it is provable iff the sequent $\Gamma, \tup(\vec{n}),p, \FORComp_\beta \focusedNeg \emptyZone$ is provable in the focused calculus. If $\beta=0$, then the latter sequent can be obtained by means of no rule, thus is underivable. This proves the base case. For $\beta>0$, each derivation of this sequent ends as follows for some $\gamma < \beta$.
\[
    \vlderivation{
    \vlin{D}{}
        {
            \Gamma, \tup(\vec{n}),p, \FORComp_\beta \focusedNeg \emptyZone
        }
        {
            \vlin{\muord \beta}{}
            {
                \Gamma, \tup(\vec{n}),p \focusedPos \FORComp_\beta
            }
            {
                \vlhy{
                    \Gamma, \tup(\vec{n}),p \focusedPos T \vlor \bigoplus\limits_{j=1}^l\bigoplus\limits_{\ins \in I_j} \encIns{\ins}^\prime_\gamma
                }
            }
        }
    }
\]
Here, $\encIns{\ins}^\prime_\gamma$ denotes the result of substituting $\FORComp_\gamma$ for $\UniNumA$ in $\encIns{\ins}$.

Note that, in the decision rule application, no formula from $\Gamma$ could be in the right zone instead of $\FORComp$. Indeed, since $\Gamma$ is $\{t,\VARkey_1\}$-locked, this would imply that the sequent $\Gamma, \tup(\vec{n}),p, \FORComp$ contains either $t$, $\VARkey_1$, or $\VARacc$, which is not the case.

The above derivation can be continued in two ways.

\textit{Case 1.} The derivation proceeds as follows for some $j \in \{1,\ldots,l\}$.
\[
\vlderivation{
\vliq{\vlor}{}
{
    \Gamma, \tup(\vec{n}),p \focusedPos T \vlor \bigoplus\limits_{j=1}^l\bigoplus\limits_{\ins \in I_j} \encIns{\ins}^\prime_\gamma
}
{
    \vliin{\vlte}{}
    {
        \Gamma, \tup(\vec{n}),p \focusedPos \neg t_j \vlte t
    }
    {
        \vlhy{p \focusedPos \neg t_j}
    }
    {
        \vliq{}{}
        {
            \Gamma, \tup(\vec{n}) \focusedPos t
        }
        {
            \vlhy{\Gamma, \tup(\vec{n}), t \focusedNeg \cdot}
        }
    }
}
}
\]
In order for $p \focusedPos \neg t_j$ to be provable, $p$ must equal $t_j$; since $p \in Q_d$, it holds, moreover, that $d=j$.

\textit{Case 2.} The derivation proceeds as follows for some $j \in \{1,\ldots,l\}$ and some $\ins \in I_j$. 
\[
\vliqf{\vlor}{}
{
    \Gamma, \tup(\vec{n}),p \focusedPos T \vlor \bigoplus\limits_{j=1}^l\bigoplus\limits_{\ins \in I_j} \encIns{\ins}^\prime_\gamma
}
{
    \Gamma, \tup(\vec{n}),p \focusedPos \encIns{\ins}^\prime_\gamma
}
\]
Two subcases arise depending on whether $\ins$ is an increment instruction or a decrement one.

\textit{Case 2a.} $\ins = \INC(p,i,r)$. Then, the derivation goes on as follows.
\[
\vlderivation{
\vliin{\vlte}{}
{
    \Gamma, \tup(\vec{n}),p \focusedPos \encIns{\ins}^\prime_\gamma
}
{
    \vlhy{p \focusedPos \neg p}
}
{
    \vliq{}{}
    {
        \Gamma, \tup(\vec{n}) \focusedPos \VARsym_i \vlpa q \vlpa \FORComp_\gamma
    }
    {
        \vlhy{\Gamma, \tup(\vec{n}^\prime) , q , \FORComp_\gamma \focusedNeg}
    }
}}
\]
Here, $\tup(\vec{n}^\prime)$ is obtained from $\tup(\vec{n})$ by adding one more occurrence of $c_i$, which means that $\vec{n}^\prime$ is obtained from $\vec{n}$ by increasing the $i$-th counter by $1$. Note that, since $p \in Q_d$ and states of  $\MinMach_1,\ldots,\MinMach_l$ are disjoint, $\ins$ belongs to $Q_d$. 

\textit{Case 2b.} $\ins = \JZDEC(p,i,q,r)$. Then, the derivation goes on as follows.
\[
\vliinf{\vlte}{}
{
    \Gamma, \tup(\vec{n}),p \focusedPos \encIns{\ins}^\prime_\gamma
}
{
    p \focusedPos \neg p
}
{
    \Gamma, \tup(\vec{n}) \focusedPos (\neg \VARsym_i  \vlte (q \vlpa \FORComp_\gamma)) \vlor (Z_i \vlan (r \vlpa \FORComp_\gamma))
}
\]
Again, two subcases arise depending on which disjunct is chosen in the distinguished formula of the topmost sequent.

\textit{Case 2b(i).} 
\[
\vlderivation{
\vlin{\vlor}{}
{
    \Gamma, \tup(\vec{n}) \focusedPos (\neg \VARsym_i  \vlte (q \vlpa \FORComp_\gamma)) \vlor (Z_i \vlan (r \vlpa \FORComp_\gamma))
}
{
    \vliin{\vlte}{}
    {
        \Gamma, \tup(\vec{n}) \focusedPos \neg \VARsym_i  \vlte (q \vlpa \FORComp_\gamma)
    }
    {
        \vlhy{\VARsym_i \focusedPos \neg \VARsym_i}
    }
    {
        \vliq{}{}
        {
            \Gamma, \tup(\vec{n}^\prime) \focusedPos q \vlpa \FORComp_\gamma
        }
        {
            \vlhy{\Gamma, \tup(\vec{n}^\prime), q, \FORComp_\gamma \focusedNeg \emptyZone}
        }
    }
}}
\]
Here $\vec{n}^\prime$ is obtained from $\vec{n}$ by decrasing the $i$-th component by $1$.

\textit{Case 2b(ii).}
\[
\vlderivation{
\vlin{\vlor}{}
{
    \Gamma, \tup(\vec{n}) \focusedPos (\neg \VARsym_i  \vlte (q \vlpa \FORComp_\gamma)) \vlor (Z_i \vlan (r \vlpa \FORComp_\gamma))
}
{
    \vlin{R}{}
    {
        \Gamma, \tup(\vec{n}) \focusedPos Z_i \vlan (r \vlpa \FORComp_\gamma)
    }
    {
        \vliin{\vlan}{}
        {
            \Gamma, \tup(\vec{n}) \focusedNeg Z_i \vlan (r \vlpa \FORComp_\gamma)
        }
        {
            \vlhy{\Gamma, \tup(\vec{n}) \focusedNeg Z_i}
        }
        {
            \vliq{}{}
            {
                \Gamma, \tup(\vec{n}) \focusedNeg r \vlpa \FORComp_\gamma
            }
            {
                \vlhy{\Gamma, \tup(\vec{n}), r, \FORComp_\gamma \focusedNeg \emptyZone}
            }
        }
    }
}}
\]
By analyzing focused derivations of the sequent 
\begin{equation}\label{sequent:FORComp-proof-1}
    \Gamma, \tup(\vec{n}) \focusedPos Z_i,
\end{equation}
a straightforward induction shows that the latter is provable iff so is a sequent of the form
\begin{equation}\label{sequent:FORComp-proof-2}
    \Gamma, \tup(\vec{n}^{\prime\prime}) , \VARacc, \neg \VARacc \focusedNeg \emptyZone
\end{equation}
for some $\vec{n}^{\prime\prime}$ which is obtained from $\vec{n}$ by decreasing some counters except for the $i$-th one. Furthermore, any focused derivation of (\ref{sequent:FORComp-proof-2}) ends as follows for some $A \in \Gamma$ (recall that $A$ is $\{t,\VARkey_1\}$-locked):
\[
\vlderivation{
\vlin{D}{}
{
    \Gamma, \tup(\vec{n}^{\prime\prime}) , \VARacc, \neg \VARacc \focusedNeg \emptyZone
}
{
    \vliq{\oplus}{}
    {
        \Gamma^\prime, \tup(\vec{n}^{\prime\prime}) , \VARacc , \neg \VARacc \focusedPos A
    }
    {
        \vliin{\vlte}{}
        {
            \Gamma^\prime, \tup(\vec{n}^{\prime\prime}) , \VARacc, \neg \VARacc \focusedPos \FORPass
        }
        {
            \vlhy{\VARacc \focusedPos \neg \VARacc}
        }
        {
            \vliq{}{}
            {
                \Gamma^\prime, \tup(\vec{n}^{\prime\prime}), \neg \VARacc \focusedPos \VARacc
            }
            {
                \vlhy{\Gamma^\prime, \tup(\vec{n}^{\prime\prime}), \VARacc, \neg \VARacc \focusedNeg \emptyZone}
            }
        }
    }
}}
\]
Here $\Gamma^\prime$ is obtained from $\Gamma$ by removing $A$. Thus, eventually, the derivation reaches the sequent $\tup(\vec{n}^{\prime\prime}), \VARacc, \neg \VARacc \focusedNeg \emptyZone$. The latter sequent is provable iff $\tup(\vec{n}^{\prime\prime})$ is empty, i.e.~$\vec{n}^{\prime\prime} = 0$. So, (\ref{sequent:FORComp-proof-1}) is provable iff one can obtain the zero vector from $\vec{n}$ by decreasing some counters except for the $i$-th one. Equivalently, this means that the $i$-th counter in $\vec{n}$ equals $0$.

Summarizing the above reasonings, we conclude that the sequent 
\[
\Gamma, \tup(\vec{n}),p, \FORComp_\beta,
\]
i.e.~the sequent (\ref{sequent:FORComp-beta}), is provable in one of the two cases:
\begin{enumerate}
    \item $\beta>0$, $p=t_d$, and $\Gamma, \tup(\vec{n}), t$ is provable;
    \item there exist $\gamma<\beta$, a vector $\vec{n}^\prime$ and a state $p^\prime$ such that 
    \[
        p(\vec{n},0,0,\ldots) \reach{\MinMach_d} p^\prime(\vec{n}^\prime,0,0,\ldots)
    \]
    and such that the sequent $\Gamma, \tup(\vec{n}^\prime),p^\prime, \FORComp_\gamma$ is provable.
\end{enumerate}
The induction hypothesis can be applied to the second statement, which implies that $\Gamma, \tup(\vec{n}^\prime),p^\prime, \FORComp_\gamma$ is provable iff a sequent of the form $\Gamma,\tup(\vec{m}),t$ is provable such that $p(\vec{n}^\prime,0,0,\ldots) \reach{\MinMach_d}^h t_d(\vec{m},0,0,\ldots)$ for $h < \gamma$. 

Combining the two cases, we obtain the desired statement.

\subsection{Proof of \cref{lemma:FORInd}}\label{proof:lemma:FORInd}

The sequent $\Gamma, \VARsym_0^\UniNumA, \FORInd$ is provable in $\muMALLord\alpha$ iff so is $\Gamma, \VARsym_0^\UniNumA, \FORInd  \focusedNeg \emptyZone$ in $\muMALLFord\alpha$. Each derivation of the latter must end by a decision rule application. In the premise of such a rule, no formula from $\Gamma$ can be chosen because this would imply that the left zone of the sequent contains $\VARkey_1$ or $\VARacc$. Thus, the chosen formula is $\FORInd$, and the premise equals $\Gamma, \VARsym_0^\UniNumA \focusedPos \FORInd$. Two cases arise depending on what the premise above that sequent is.

\textit{Case 1.} The premise is $\Gamma, \VARsym_0^\UniNumA \focusedPos \FORBase$. Any its derivation ends as follows.
\[
\vlderivation{
\vlin{R}{}
{
    \Gamma, \VARsym_0^\UniNumA \focusedPos \FORBase
}
{
    \vliq{}{}
    {
        \Gamma, \VARsym_0^\UniNumA \focusedNeg \FORBase 
    }
    {
        \vlhy{\Gamma, \VARsym_0^\UniNumA, s_1 , \FORComp , t^\bot \vlte \FORAcc  \focusedNeg \emptyZone}
    }
}}
\]
We can readily apply \cref{lemma:FORComp} to the sequent $\Gamma, \VARsym_0^\UniNumA, s_1 , \FORComp , t^\bot \vlte \FORAcc$ and conclude that it is derivable iff $f_1(\UniNumA)$ is defined and the sequent
\[
\Gamma, \tup(f_1(\UniNumA)) , t , t^\bot \vlte \FORAcc
\]
is derivable. $f_1(\UniNumA)$ is defined iff $\ips{\UniNumA}=\top$. In that case, $f_1(\UniNumA)=0$ and the latter sequent has the form
\[
\Gamma, t , t^\bot \vlte \FORAcc
\]
It is provable because each formula in $\Gamma$ is a disjunction with $\FORPass = \neg \VARacc \vlte \VARacc$ being one of the disjuncts, and $\FORAcc = \neg \VARacc \vlpa \VARacc$.

\textit{Case 2.} The premise is $\Gamma, \VARsym_0^\UniNumA \focusedPos \FORStep$. Any its derivation ends as follows.
\[
\vlderivation{
\vlin{R}{}
{
    \Gamma, \VARsym_0^\UniNumA \focusedPos \FORStep
}
{
    \vliq{}{}
    {
        \Gamma, \VARsym_0^\UniNumA \focusedNeg \FORStep
    }
    {
        \vlhy{\Gamma, \VARsym_0^\UniNumA, s_2 , \FORComp , t^\bot \vlte ((\VARsym_1^\bot \vlte \FORStep_\exists)\vlor (\VARsym_2^\bot \vlte \FORStep_\forall))  \focusedNeg \emptyZone}
    }
}}
\]
By \cref{lemma:FORComp}, the sequent
\[
\Gamma, \VARsym_0^\UniNumA, s_2 , \FORComp , t^\bot \vlte ((\VARsym_1^\bot \vlte \FORStep_\exists)\vlor (\VARsym_2^\bot \vlte \FORStep_\forall))  \focusedNeg \emptyZone
\]
is derivable iff $f_2(\UniNumA)$ is defined and the sequent
\begin{equation}\label{sequent:FORInd-1}
    \Gamma, \tup(f_2(\UniNumA)), t , t^\bot \vlte ((\VARsym_1^\bot \vlte \FORStep_\exists)\vlor (\VARsym_2^\bot \vlte \FORStep_\forall))  \focusedNeg \emptyZone
\end{equation}
is derivable. Now, two further cases arise.

\textit{Case 2a.} $\UniNumA \in S^\Sigma_a$ for $a \in \KleeneO$ such that $1 \lessO a \lessO a_0$. Then, $f_2(\UniNumA) = (\UniNumA,1,0)$, hence $\tup(f_2(\UniNumA)) = \VARsym_0^\UniNumA,\VARsym_1$. Any derivation of (\ref{sequent:FORInd-1}) ends as follows.
\[
\vlderivation{
\vlin{D}{}
{
    \Gamma, \VARsym_0^\UniNumA,\VARsym_1, t , t^\bot \vlte ((\VARsym_1^\bot \vlte \FORStep_\exists)\vlor (\VARsym_2^\bot \vlte \FORStep_\forall))  \focusedNeg \emptyZone
}
{
    \vliin{\vlte}{}
    {
        \Gamma, \VARsym_0^\UniNumA,\VARsym_1, t \focusedPos t^\bot \vlte ((\VARsym_1^\bot \vlte \FORStep_\exists)\vlor (\VARsym_2^\bot \vlte \FORStep_\forall))
    }
    {
        \vlhy{t \focusedPos t^\bot}
    }
    {
        \vlin{\vlor}{}
        {
            \Gamma, \VARsym_0^\UniNumA,\VARsym_1 \focusedPos (\VARsym_1^\bot \vlte \FORStep_\exists)\vlor (\VARsym_2^\bot \vlte \FORStep_\forall)
        }
        {
            \vliin{\vlte}{}
            {
                \Gamma, \VARsym_0^\UniNumA,\VARsym_1 \focusedPos \VARsym_1^\bot \vlte \FORStep_\exists
            }
            {
               \vlhy{ \VARsym_1 \focusedPos \VARsym_1^\bot}
            }
            {
                \vliq{}{}
                {
                    \Gamma, \VARsym_0^\UniNumA \focusedPos \FORStep_\exists
                }
                {
                    \vlhy{\Gamma, \VARsym_0^\UniNumA, \mu x. ((\VARsym_1 \vlpa x) \vlor (\VARsym_2 \vlpa x) \vlor \VARkey_2) , \FORNext_\exists  \focusedNeg \emptyZone}
                }
            }
        }
    }
}}
\]
(Recall that $\Gamma$ is $\{\VARkey_1\}$-locked which prevents its formulae from being moved to the right zone.) 

Our goal now is to analyse derivability of the sequent
\begin{equation}\label{sequent:FORStep_exists}
    \Gamma, \VARsym_0^\UniNumA, \mu x. ((\VARsym_1 \vlpa x) \vlor (\VARsym_2 \vlpa x) \vlor \VARkey_2) , \FORNext_\exists  \focusedNeg \emptyZone
\end{equation}
Let us denote the formula $\muord \beta x. ((\VARsym_1 \vlpa x) \vlor (\VARsym_2 \vlpa x) \vlor \VARkey_2)$ by $M_\beta$. Consider a sequent of the form 
\[
\Gamma, \VARsym_0^\UniNumA, \VARsym_1^\UniNumB , \VARsym_2^\UniNumC , M_\beta, \FORNext_\exists  \focusedNeg \emptyZone
\]
We shall call this sequent the $(\UniNumB,\UniNumC,M_\beta)$-sequent. Any of its derivations ends as follows for some $\gamma < \beta$:
\[
\vlderivation{
\vlin{D}{}
{
    \Gamma, \VARsym_0^\UniNumA, \VARsym_1^\UniNumB , \VARsym_2^\UniNumC , M_\beta, \FORNext_\exists  \focusedNeg \emptyZone
}
{
    \vlin{\muord \beta}{}
    {
        \Gamma, \VARsym_0^\UniNumA, \VARsym_1^\UniNumB , \VARsym_2^\UniNumC , \FORNext_\exists \focusedPos M_\beta
    }
    {
        \vlhy{\Gamma, \VARsym_0^\UniNumA, \VARsym_1^\UniNumB , \VARsym_2^\UniNumC , \FORNext_\exists \focusedPos (\VARsym_1 \vlpa M_\gamma) \vlor (\VARsym_2 \vlpa M_\gamma) \vlor \VARkey_2}
    }
}}
\]
Then, one chooses on of the three disjuncts and, as the derivation continues from bottom to top, one obtains either the $(\UniNumB+1,\UniNumC,M_\gamma)$-sequent, the $(\UniNumB,\UniNumC+1,M_\gamma)$-sequent, or the sequent
\begin{equation}\label{sequent:FORStep_exists-2}
    \Gamma, \VARsym_0^\UniNumA, \VARsym_1^\UniNumB , \VARsym_2^\UniNumC , \FORNext_\exists, \VARkey_2  \focusedNeg \emptyZone
\end{equation}
Note that (\ref{sequent:FORStep_exists}) is the $(0,0,M_\alpha)$-sequent. The above implies that (\ref{sequent:FORStep_exists}) is derivable iff so is (\ref{sequent:FORStep_exists-2}) for some $\UniNumB,\UniNumC \in \omega$. Any derivation of the latter ends as follows.
\[
\vlderivation{
\vlin{}{}
{
    \Gamma, \VARsym_0^\UniNumA, \VARsym_1^\UniNumB , \VARsym_2^\UniNumC , \FORNext_\exists, \VARkey_2  \focusedNeg \emptyZone
}
{
    \vliin{}{}
    {
        \Gamma, \VARsym_0^\UniNumA, \VARsym_1^\UniNumB , \VARsym_2^\UniNumC , \VARkey_2 \focusedPos \VARkey_2^\bot \vlte (s_3 \vlpa \FORComp \vlpa (t^\bot \vlte \VARsym_1^\bot \vlte \VARkey_1))
    }
    {
       \vlhy{ \VARkey_2 \focusedPos \VARkey_2^\bot}
    }
    {
        \vliq{}{}
        {
            \Gamma, \VARsym_0^\UniNumA, \VARsym_1^\UniNumB , \VARsym_2^\UniNumC \focusedPos s_3 \vlpa \FORComp \vlpa (t^\bot \vlte \VARsym_1^\bot \vlte \VARkey_1)
        }
        {
            \vlhy{\Gamma, \VARsym_0^\UniNumA, \VARsym_1^\UniNumB , \VARsym_2^\UniNumC, s_3 , \FORComp , t^\bot \vlte \VARsym_1^\bot \vlte \VARkey_1  \focusedNeg \emptyZone}
        }
    }
}}
\]
The topmost sequent in the above derivation, by \cref{lemma:FORComp}, is derivable iff $f_3(\UniNumA,\UniNumB,\UniNumC)$ is defined and the sequent
\[
\Gamma, \tup(f_3(\UniNumA,\UniNumB,\UniNumC)), t, t^\bot \vlte \VARsym_1^\bot \vlte \VARkey_1  \focusedNeg \emptyZone
\]
is derivable. Equivalently, one of the two holds:
\begin{enumerate}[i.]
    \item $\Witn(\UniNumA,\UniNumB,\UniNumC)$ is true and the sequent $\Gamma, \VARsym_0^\UniNumB,\VARsym_1, t, t^\bot \vlte \VARsym_1^\bot \vlte \VARkey_1  \focusedNeg \emptyZone$ is provable. 
    \item $\Witn(\UniNumA,\UniNumB,\UniNumC)$ is false and the sequent $\Gamma, \VARsym_2, t, t^\bot \vlte \VARsym_1^\bot \vlte \VARkey_1  \focusedNeg \emptyZone$ is provable. 
\end{enumerate}
It is easy to check that the second sequent is not provable because there are no occurrences of $\VARsym_1$ in it. Any derivation of the first one ends as follows:
\[
\vlderivation{
\vlin{D}{}
{
    \Gamma, \VARsym_0^\UniNumB,\VARsym_1, t, t^\bot \vlte \VARsym_1^\bot \vlte \VARkey_1  \focusedNeg \emptyZone
}
{
    \vliin{\vlte}{}
    {
        \Gamma, \VARsym_0^\UniNumB,\VARsym_1, t \focusedPos t^\bot \vlte \VARsym_1^\bot \vlte \VARkey_1
    }
    {
        \vlhy{t \focusedPos t^\bot}
    }
    {
        \vliin{\vlte}{}
        {
            \Gamma, \VARsym_0^\UniNumB,\VARsym_1 \focusedPos \VARsym_1^\bot \vlte \VARkey_1
        }
        {
            \vlhy{\VARsym_1 \focusedPos \VARsym_1^\bot}
        }
        {
            \vliq{}{}
            {
                \Gamma, \VARsym_0^\UniNumB \focusedPos \VARkey_1
            }
            {
                \vlhy{\Gamma, \VARsym_0^\UniNumB , \VARkey_1  \focusedNeg \emptyZone}
            }
        }
    }
}}
\]
Let us conclude Case 2a. Given that $\UniNumA \in S^\Sigma_a$ for $a \in \KleeneO$ such that $1 \lessO a \lessO a_0$, the sequent $\Gamma,\VARsym_0^\UniNumA,\FORInd$ is derivable $\Longleftrightarrow$ (\ref{sequent:FORInd-1}) is derivable $\Longleftrightarrow$ there is $\UniNumB \in \omega$ such that $\exists \UniNumC\, \Witn(\UniNumA,\UniNumB,\UniNumC)$ holds and the sequent $\Gamma, \VARsym_0^\UniNumB , \VARkey_1$ is derivable.

\textit{Case 2b.} $\UniNumA \in S^\Pi_a$ for $a \in \KleeneO$ such that $1 \lessO a \lessO a_0$. Then, $f_2(\UniNumA) = (\UniNumA,0,1)$, hence $\tup(f_2(\UniNumA)) = \VARsym_0^\UniNumA,\VARsym_2$. In such case, any derivation of (\ref{sequent:FORInd-1}) ends as follows.
\[
\vlderivation{
\vlin{D}{}
{
    \Gamma, \VARsym_0^\UniNumA,\VARsym_2, t , t^\bot \vlte ((\VARsym_1^\bot \vlte \FORStep_\exists)\vlor (\VARsym_2^\bot \vlte \FORStep_\forall))  \focusedNeg \emptyZone
}
{
    \vliin{\vlte}{}
    {
        \Gamma, \VARsym_0^\UniNumA,\VARsym_2, t \focusedPos t^\bot \vlte ((\VARsym_1^\bot \vlte \FORStep_\exists)\vlor (\VARsym_2^\bot \vlte \FORStep_\forall))
    }
    {
        \vlhy{t \focusedPos t^\bot}
    }
    {
        \vlin{\vlor}{}
        {
            \Gamma, \VARsym_0^\UniNumA,\VARsym_2 \focusedPos (\VARsym_1^\bot \vlte \FORStep_\exists)\vlor (\VARsym_2^\bot \vlte \FORStep_\forall)
        }
        {
            \vliin{\vlte}{}
            {
                \Gamma, \VARsym_0^\UniNumA,\VARsym_2 \focusedPos \VARsym_2^\bot \vlte \FORStep_\forall
            }
            {
                \vlhy{\VARsym_2 \focusedPos \VARsym_2^\bot}
            }
            {
                \vliq{}{}
                {
                    \Gamma, \VARsym_0^\UniNumA \focusedPos \FORStep_\forall
                }
                {
                    \vlhy{\Gamma, \VARsym_0^\UniNumA, \FORNext_\forall \focusedNeg \nu x. ((\VARsym_1 \vlpa x) \vlan (\VARsym_2 \vlpa x) \vlan \VARkey_2)}
                }
            }
        }
    }
}}
\]
Let us denote the formula $\nuord \beta x. ((\VARsym_1 \vlpa x) \vlan (\VARsym_2 \vlpa x) \vlan \VARkey_2)$ by $N_\beta$ and let us call the sequent 
\[
\Gamma, \VARsym_0^\UniNumA, \VARsym_1^\UniNumB, \VARsym_2^\UniNumC, \FORNext_\forall \focusedNeg N_\beta 
\]
the $(\UniNumB,\UniNumC,N_\beta)$-sequent. (The topmost sequent in the above derivation is the $(0,0,N_\alpha)$-sequent.) Any derivation of the $(\UniNumB,\UniNumC,N_\beta)$-sequent ends as follows:
\[
\vlderivation{
\vlin{\nuord \beta}{}
{
    \Gamma, \VARsym_0^\UniNumA, \VARsym_1^\UniNumB, \VARsym_2^\UniNumC, \FORNext_\forall \focusedNeg N_\beta , \FORNext_\forall
}
{
    {\vliiiq{}{}
    {
        \{\Gamma, \VARsym_0^\UniNumA, \VARsym_1^\UniNumB, \VARsym_2^\UniNumC, \FORNext_\forall \focusedNeg (\VARsym_1 \vlpa N_\gamma) \vlan (\VARsym_2 \vlpa N_\gamma) \vlan \VARkey_2\}_{\gamma<\beta}
    }
    {
        \vlhy{\{\Gamma, \VARsym_0^\UniNumA, \VARsym_1^{\UniNumB+1}, \VARsym_2^\UniNumC, \FORNext_\forall \focusedNeg N_\gamma \}_{\gamma<\beta}}
    }
    {
        \vlhy{\{\Gamma, \VARsym_0^\UniNumA, \VARsym_1^\UniNumB, \VARsym_2^{\UniNumC+1}, \FORNext_\forall \focusedNeg N_\gamma \}_{\gamma<\beta}}
    }
    {
        \vlhy{\{\Gamma, \VARsym_0^\UniNumA, \VARsym_1^\UniNumB, \VARsym_2^\UniNumC, \FORNext_\forall , \VARkey_2   \focusedNeg \emptyZone\}_{\gamma<\beta}}
    }}
}}
\]
This implies that the $(\UniNumB,\UniNumC,N_\beta)$-sequent is derivable iff so are the $(\UniNumB+1,\UniNumC,N_\gamma)$-sequent, the $(\UniNumB,\UniNumC+1,N_\gamma)$-sequent and the sequent $\Gamma, \VARsym_0^\UniNumA, \VARsym_1^\UniNumB, \VARsym_2^\UniNumC , \FORNext_\forall , \VARkey_2   \focusedNeg \emptyZone$, for all $\gamma<\beta$. Consequently, the $(0,0,N_\alpha)$-sequent is derivable iff so is 
\begin{equation}\label{sequent:FORNext_forall-2}
    \Gamma, \VARsym_0^\UniNumA, \VARsym_1^\UniNumB, \VARsym_2^\UniNumC , \VARkey_2 , \FORNext_\forall  \focusedNeg \emptyZone
\end{equation}
for all $\UniNumB,\UniNumC \in \omega$. Reasoning in the same way as when analysing derivations of (\ref{sequent:FORStep_exists-2}) we deduce that (\ref{sequent:FORNext_forall-2}) is provable iff one of the two holds:
\begin{enumerate}[i.]
    \item $\Witn(\UniNumA,\UniNumB,\UniNumC)$ is true and $\Gamma, \VARsym_0^\UniNumB,\VARsym_1, t, t^\bot \vlte ((\VARsym_1^\bot \vlte \VARkey_1) \vlor (\VARsym_2^\bot \vlte \FORAcc))  \focusedNeg \emptyZone$ is provable. 
    \item $\Witn(\UniNumA,\UniNumB,\UniNumC)$ is false and $\Gamma, \VARsym_2, t, t^\bot \vlte ((\VARsym_1^\bot \vlte \VARkey_1) \vlor (\VARsym_2^\bot \vlte \FORAcc))  \focusedNeg \emptyZone$ is provable. 
\end{enumerate}
In Case i, the sequent's derivation ends as follows:
\[
\vlderivation{
\vlin{D}{}
{
    \Gamma, \VARsym_0^\UniNumB,\VARsym_1, t, t^\bot \vlte ((\VARsym_1^\bot \vlte \VARkey_1) \vlor (\VARsym_2^\bot \vlte \FORAcc))  \focusedNeg \emptyZone
}
{
    \vliin{\vlte}{}
    {
        \Gamma, \VARsym_0^\UniNumB,\VARsym_1, t \focusedPos t^\bot \vlte ((\VARsym_1^\bot \vlte \VARkey_1) \vlor (\VARsym_2^\bot \vlte \FORAcc))
    }
    {
        \vlhy{t \focusedPos t^\bot}
    }
    {
        \vlin{\vlor}{}
        {
             \Gamma, \VARsym_0^\UniNumB,\VARsym_1 \focusedPos (\VARsym_1^\bot \vlte \VARkey_1) \vlor (\VARsym_2^\bot \vlte \FORAcc)
        }
        {
            \vliin{\vlte}{}
            {
                \Gamma, \VARsym_0^\UniNumB,\VARsym_1 \focusedPos \VARsym_1^\bot \vlte \VARkey_1
            }
            {
               \vlhy{ \VARsym_1 \focusedPos \VARsym_1^\bot}
            }
            {
                \vliq{}{}
                {
                    \Gamma, \VARsym_0^\UniNumB \focusedPos \VARkey_1
                }
                {
                    \vlhy{\Gamma, \VARsym_0^\UniNumB , \VARkey_1  \focusedNeg \emptyZone}
                }
            }
        }
    }
}}
\]
In Case ii, the sequent's derivation ends as follows:
\[
\vlderivation{
\vlin{D}{}
{
    \Gamma, \VARsym_2, t, t^\bot \vlte ((\VARsym_1^\bot \vlte \VARkey_1) \vlor (\VARsym_2^\bot \vlte \FORAcc))  \focusedNeg \emptyZone
}
{
    \vliin{\vlte}{}
    {
        \Gamma, \VARsym_2, t \focusedPos t^\bot \vlte ((\VARsym_1^\bot \vlte \VARkey_1) \vlor (\VARsym_2^\bot \vlte \FORAcc))
    }
    {
        \vlhy{t \focusedPos t^\bot}
    }
    {
        \vlin{\vlor}{}
        {
             \Gamma, \VARsym_2 \focusedPos (\VARsym_1^\bot \vlte \VARkey_1) \vlor (\VARsym_2^\bot \vlte \FORAcc)
        }
        {
            \vliin{\vlte}{}
            {
                \Gamma, \VARsym_2 \focusedPos \VARsym_2^\bot \vlte \FORAcc
            }
            {
                \vlhy{\VARsym_2 \focusedPos \VARsym_2^\bot}
            }
            {
                \vliq{}{}
                {
                    \Gamma \focusedPos  \FORAcc
                }
                {
                    \vlhy{\Gamma, \neg \VARacc, \VARacc  \focusedNeg \emptyZone}
                }
            }
        }
    }
}}
\]
The sequent $\Gamma, \neg \VARacc, \VARacc  \focusedNeg \emptyZone$ is derivable because each formula in $\Gamma$ is a disjunction with $\FORPass$ being one of disjuncts. Therefore, the sequent in Case ii is always derivable.

Let us conclude Case 2b. Given that $\UniNumA \in S^\Pi_a$ for $a \in \KleeneO$ such that $1 \lessO a \lessO a_0$, the sequent $\Gamma,\VARsym_0^\UniNumA,\FORInd$ is derivable $\Longleftrightarrow$ (\ref{sequent:FORInd-1}) is derivable $\Longleftrightarrow$ for each $\UniNumB \in \omega$ such that $\exists \UniNumC\, \Witn(\UniNumA,\UniNumB,\UniNumC)$, the sequent $\Gamma, \VARsym_0^\UniNumB , \VARkey_1$ is derivable.

It remains to recall that the initial sequent $\Gamma,\VARsym_0^\UniNumA,\FORInd$ is provable iff either Case 1, Case 2a, or Case 2b takes place. This proves the lemma.

\section{Conclusions}
\label{sec:concs}

In this work we investigated systems $\muMALLord\alpha$ for linear logic with least and greatest fixed points parametrised by a closure ordinal $\alpha$.
We developed its proof theoretic foundations, namely proving cut-elimination and focussing results, and applied them to classify the complexity of these logics for computable $\alpha$: $\muMALLord\alpha$ is complete for the $\omega^{\alpha^\omega}$ level of the hyperarithmetical hierarchy (under Turing reductions).

For $\alpha = \omega$, the system $\muMALLord\omega$ was previously studied in \cite{Jafarrahmani2021CSL} where a sound and complete phase semantics was given.
In this vein an interesting special case of our results is:

\begin{corollary}
    If $\omega \le \alpha < \omega^\omega$, then provability in $\muMALLord \alpha$ is $\Sigma^0_{\omega^{\omega^\omega}}$-complete. In particular, the provability problems for all such $\alpha$ are Turing-equivalent. 
\end{corollary}

\begin{proof}
    $1 \le \deg \alpha < \omega$ for such $\alpha$, therefore, $\omega^{\alpha^\omega} = \omega^{\omega^{(\deg\alpha) \omega}} = \omega^{\omega^\omega}$. The result follows from Theorem \ref{theorem:complexity-main}.
\end{proof}

In a related direction, \cite{KuznetsovS22,Pshenitsyn24} investigate the complexity of \textit{infinitary action logic with multiplexing}.
This logic differs from $\muMALLord\omega$ in several respects. 
It is based on an intuitionistic non-commutative fragment of $\MALL$ with certain fixed points in the form of a multiplexing subexponential and Kleene $*$.
This logic turns out to be $\Sigma^0_{\omega^\omega}$-complete, a whole $\omega$-tower lower than for $\muMALLord\omega$.
It would be interesting to understand more generally how combinations of logical features relate to hyperarithmetical complexity.

\section*{Acknowledgments}

The authors thank Abhishek De, Stepan Kuznetsov and Stanislav Speranski for helpful discussions about this and related work.
The alphabetically first author is supported by a UKRI Future Leaders fellowship, \emph{Structure vs Invariants in Proofs}, project number MR/S035540/1. The second author's work was performed at the Steklov International Mathematical Center and supported by the Ministry of Science and Higher Education of
the Russian Federation (agreement no. 075-15-2025-303).

\tikhon{there are 27 refs while the submission contains 31 entries. shouls we include those 4 refs as well?}

\bibliographystyle{plain}
\bibliography{bibliography}

\end{document}